\documentclass[12pt]{amsart}%
\usepackage{amsmath}%
\usepackage{amssymb}%
\usepackage{amscd}%
\usepackage{color}%
\usepackage{graphicx}%
\usepackage{mathrsfs}%
\usepackage[all]{xy}%
\usepackage{newtxtext,newtxmath}%
\usepackage[all, warning]{onlyamsmath}
\usepackage{amsrefs}
\usepackage{bbm}
\def\ol#1{\overline{#1}}% 		overline
\def\wh#1{\widehat{#1}}% 	wide hat
\def\wt#1{\widetilde{#1}}% 	wide tilde
\theoremstyle{plain}
    \newtheorem{theorem}{Theorem}[section]
    \newtheorem{proposition}[theorem]{Proposition}
    \newtheorem{lemma}[theorem]{Lemma}
    \newtheorem{corollary}[theorem]{Corollary}
      
\theoremstyle{definition}
    \newtheorem{definition}[theorem]{Definition}

    \newtheorem{remark}[theorem]{Remark}
    
\def\Alphabet{A,B,C,D,E,F,G,H,I,J,K,L,M,N,O,P,Q,R,S,T,U,V,W,X,Y,Z}%  Capitalized Alphabet
\def\alphabet{a,b,c,d,e,f,g,h,i,j,k,l,m,n,o,p,q,r,s,t,u,v,w,x,y,z}%	lowercase alphabet
\def\endpiece{xxx}%									marks end of list
\def\makeAlphabet[#1]{\expandafter\makeA#1,xxx,}%		Ex. \makeAlphabet[A,B]
\def\makealphabet[#1]{\expandafter\makea#1,xxx,}%		Ex. \makealphabet[c,d]
\def\makeA#1,{\def\temp{#1}\ifx\temp\endpiece\else%
\mkbb{#1}\mkfrak{#1}\mkbf{#1}\mkcal{#1}\mkscr{#1}\mkbs{#1}\expandafter\makeA\fi}%
\def\makea#1,{\def\temp{#1}\ifx\temp\endpiece\else\mkfrak{#1}\mkbf{#1}\mkbs{#1}\expandafter\makea\fi}%
\def\mkbb#1{\expandafter\def\csname bb#1\endcsname{\mathbb{#1}}}%      Define bb
\def\mkfrak#1{\expandafter\def\csname fr#1\endcsname{\mathfrak{#1}}}%    Define frak
\def\mkbf#1{\expandafter\def\csname b#1\endcsname{\mathbf{#1}}}%           Define bold letters
\def\mkcal#1{\expandafter\def\csname c#1\endcsname{\mathcal{#1}}}%       Define calligraphy
\def\mkscr#1{\expandafter\def\csname s#1\endcsname{\mathscr{#1}}}%       Define script
\def\mkbs#1{\expandafter\def\csname bs#1\endcsname{{\boldsymbol{#1}}}}%       Define bold symbol
\def\makeop[#1]{\xmakeop#1,xxx,}%					Ex. \makeop[Hom,Spec]
\def\mkop#1{\expandafter\def\csname #1\endcsname{{\mathrm{#1}}}} % 
\def\xmakeop#1,{\def\temp{#1}\ifx\temp\endpiece\else\mkop{#1}\expandafter\xmakeop\fi}%
\def\makeup[#1]{\xmakeup#1,xxx,}%					Ex. \makeop[Hom,Spec]
\def\mkup#1{\expandafter\def\csname #1\endcsname{{\mathrm{#1}\,}}} % 
\def\xmakeup#1,{\def\temp{#1}\ifx\temp\endpiece\else\mkup{#1}\expandafter\xmakeup\fi}%
\makeAlphabet[\Alphabet]%				Define bb, frak, bf, cal for Capitalized Alphabet
\makealphabet[\alphabet]%  				Define frak and bf for uncapitalized alphabet
\makeop[Alt,End,Ext,Hom,id,GL,Tr,sgn,ab,Gal,fin,Cl,tors,Re,alt,qis,Li,cont,Ob,Sch,prim]%Homs
\makeup[Spec,Im,Ker,Coker,gen,Spf,Sp]
\def\bsalpha{{\boldsymbol{\alpha}}}
\def\bsgamma{\boldsymbol{\gamma}}
\def\e{\varepsilon}

\def\divides{\,|\,}
\def\xiDelta{\xi\kern-0.3mm\Delta}
\def\xipDelta{\xi_{\kern-0.3mm p}\kern-0.2mm\Delta}
\def\p{{(p)}}
\def\bra#1{{\langle #1\rangle}}
\def\prim{0}

\newcommand{\isomto}{\xrightarrow{\cong}}
\newcommand{\integ}{\mathrm{int}}
\begin{document}
%--- the title ------------------------------------------------------------------------------------------------------------------------

\title[$p$-adic Polylogarithm]{$p$-adic Polylogarithms and $p$-adic Hecke $L$-functions 
for Totally Real Fields}
\author[Bannai]{Kenichi Bannai$^{*\diamond}$}
\email{bannai@math.keio.ac.jp}
\address{${}^*$Department of Mathematics, Faculty of Science and Technology, Keio University, 3-14-1 Hiyoshi, Kouhoku-ku, Yokohama 223-8522, Japan}
\address{${}^\diamond$Mathematical Science Team, RIKEN Center for Advanced Intelligence Project (AIP), 1-4-1 Nihonbashi, Chuo-ku, Tokyo 103-0027, Japan}
\address{${}^\dagger$Graduate School of Mathematical Sciences, The University of Tokyo, 
3-8-1 Komaba, Meguro-ku, Tokyo 153-8914, Japan}
\author[Hagihara]{Kei Hagihara$^*$}
\author[Yamada]{Kazuki Yamada$^*$}
\author[Yamamoto]{Shuji Yamamoto$^{\dagger}$}
\date{\today}

\date{\today\,\, (Version $5.0$)}
\begin{abstract}
	The purpose of this article is to newly define the $p$-adic polylogarithm 
	as an equivariant class in the cohomology of a certain infinite disjoint union of algebraic tori 
	associated to a totally real field.
	We will then express the special values of $p$-adic $L$-functions interpolating 
	nonpositive values of Hecke $L$-functions of the totally real field in terms of
	special values of these $p$-adic polylogarithms.
\end{abstract}
\thanks{This research is supported by KAKENHI 18H05233.
This research initiated from the KiPAS program FY2014--2018 of the Faculty of Science and Technology at Keio University.}

%\thanks{The author is supported in part by the JSPS postdoctoral fellowship for research abroad.}
\maketitle
\setcounter{tocdepth}{1}

%%%%%%%%%%%%%%%%%%%%%%%%%%%%%%%%%%%%%%%%%%%%%%%%%%
%
%
%
\section{Introduction}\label{section: introduction}
%
%
%
%%%%%%%%%%%%%%%%%%%%%%%%%%%%%%%%%%%%%%%%%%%%%%%%%%

The classical \textit{polylogarithm functions} $\Li_k(t)$ for integers $k\geq 0$ are
complex holomorphic functions defined by the convergent power series
\begin{equation}\label{eq: polylog}
	\Li_k(t)\coloneqq\sum_{n=1}^\infty\frac{t^n}{n^k}, \qquad
	|t|<1.
\end{equation}
Using its description in terms of iterated integrals,
the functions $\Li_k(t)$ may analytically be continued to multivalued functions on $\bbP^1\setminus\{0,1,\infty\}$.
One of the importance of the  polylogarithm functions stems from the result by Be\u{\i}linson \cite{Bei84}
(see also \cite{Neu88}*{(5.1) Theorem} or \cite{HW98}), which states that the regulator maps for Dirichlet motives 
may explicitly be described in terms of the values of the polylogarithm functions at the roots of unity.  
Combined with the fact that the special values of Dirichlet $L$-functions may also be expressed as 
linear combinations of values of the polylogarithm functions at roots of unity,
this result gives the proof of the Gross conjecture for Dirichlet $L$-functions, equivalent in this case to the 
Be\u{\i}linson conjecture for Dirichlet motives.
For Hecke $L$-functions associated to imaginary quadratic fields,
analogues of the polylogarithm functions
studied by Be\u{\i}linson and Levin \cites{BL94, Lev97, BKT10}
are given by the \textit{Kronecker-Lerch functions},
which plays an important role in the proof by Deninger \cites{Den89,Den90} 
of the Be\u{\i}linson conjecture for the associated motives.
However, natural analogues of \eqref{eq: polylog} for the case of Hecke $L$-functions for
general number fields, including the case for totally real fields, have so far been elusive.

Let $p$ be a rational prime.
The purpose of this article is to newly construct natural $p$-adic analogues of the polylogarithm functions \eqref{eq: polylog},
which we simply refer to as the \textit{$p$-adic polylogarithms},
corresponding to the case of Hecke $L$-functions associated to totally real fields.  Our construction is based on the idea 
behind the construction of the Shintani generating class \cite{BHY19}
-- that equivariant cohomology classes instead of functions give canonical objects in the higher dimensional cases.
One prominent original feature of this construction is our use of the infinite disjoint union $\bbT$ of algebraic tori
associated to the totally real field, with multiplicative action of the totally positive elements of the totally real field.
We believe $\bbT$ and its quotient stack $\sT$ should be important geometric objects in the study of the arithmetic
of totally real fields (see Remark \ref{rem: speculation}). 
We construct the $p$-adic polylogarithms as equivariant classes in the $(g-1)$-st cohomology of a certain rigid analytic subspace of $\bbT$, where $g$ is the degree of the totally real field.
We then give a construction of
the $p$-adic $L$-functions interpolating special values of Hecke $L$-functions of
totally real fields following Cassou-Nogu\`es \cite{CN79}, albeit in a more 
streamlined fashion via stronger 
emphasis on the Lerch zeta functions.
Our main theorem, Theorem \ref{thm: main}, expresses the special values of the $p$-adic Hecke $L$-functions as
a linear combination of values at torsion points of the $p$-adic polylogarithms.

In this article,
we fix once and for all embeddings $\ol\bbQ\hookrightarrow\bbC$ and $\ol\bbQ\hookrightarrow\bbC_p$.
The result we generalize is as follows.
The $p$-adic polylogarithms for the case $F=\bbQ$ are functions, originally studied by
Coleman \cite{Col82}*{VI} and Deligne \cite{Del89}*{(3.2.2)}, defined by the series
\begin{equation}\label{eq: classical polylogarithm}
	\Li^\p_k(t)\coloneqq\sum_{\substack{n=1\\(n,p)=1}}^\infty\frac{t^n}{n^k}
\end{equation}
for integers $k\in\bbZ$.
The functions $\Li^\p_k(t)$ extend to rigid analytic functions on the rigid analytic space $\wh U_{\bbQ_p}$ associated to the 
$p$-adic completion $\wh U_{\bbZ_p}$ of
the scheme $U\coloneqq\bbP^1\setminus\{0,1,\infty\}=\bbG_m\setminus\{1\}$
(see for example \cite{Col82}*{Proposition 6.2}, 
\cite{Del89}*{(3.2.3)} or Proposition \ref{prop: continuation}).
Let $N>1$ be an integer.
For any primitive Dirichlet character $\chi\colon(\bbZ/N\bbZ)^\times\rightarrow\bbC^\times$, 
Kubota and Leopoldt \cite{KL64}
defined the $p$-adic $L$-function $L_p(\chi,s)$ as an analytic function for $s\in\bbZ_p$ satisfying the interpolation
property
\[
	L_p(\chi,-k)=\big(1-\chi\omega_p^{-k-1}(p)p^{k}\big)L(\chi \omega_p^{-k-1},-k)
\]
for any integer $k\geq0$,
where $\omega_p\colon\bbZ_p^\times
\rightarrow\bbC^\times$ is the Teichm\"uller character.
Generalizing the result of Koblitz \cite{Kob79} for the case $k=1$, Coleman \cite{Col82} proved the following.

\begin{theorem} 
Suppose $N$ is not a power of $p$, and let $\xi$ be a primitive 
$N$-th root of unity. Then  for any integer $k\in\bbZ$, we have
	\[
		L_p(\chi \omega_p^{1-k},k)=\frac{g(\chi,\xi)}{N}\sum_{\beta\in(\bbZ/N\bbZ)^\times}
		\chi(\beta)^{-1}\,\Li^\p_k(\xi^\beta),
	\]
	where $g(\chi,\xi)$ is the Gauss sum
	$
		g(\chi,\xi)\coloneqq\sum_{\beta\in(\bbZ/N\bbZ)^\times} \chi(\beta)\xi^{-\beta}.
	$
\end{theorem}

In this article, we generalize the above result to the case of totally real fields.
Let $F$ be a totally real field of degree $g$ with ring of integers $\cO_F$, and let $\frI$ be the multiplicative group of nonzero 
fractional ideals of $F$.
For any $\fra$ in $\frI$, we let 
$
	\bbT^\fra\coloneqq\Hom_\bbZ(\fra,\bbG_m)
$
be the algebraic torus defined over $\bbZ$ which represents the functor 
associating to any $\bbZ$-algebra $R$ the group $\bbT^\fra(R)=\Hom_\bbZ(\fra,R^\times)$.
We let $\bbT\coloneqq\coprod_{\fra\in\frI}\bbT^\fra$
and $U\coloneqq\coprod_{\fra\in\frI}U^\fra\subset\bbT$,
where  $U^\fra\coloneqq\bbT^\fra\setminus\{1\}$.
In what follows, we fix a finite extension $K$ of $\bbQ_p$ in $\bbC_p$ containing all the conjugates of $F$.
We let
$\wh U_K\coloneqq\coprod_{\fra\in\frI}\wh U^\fra_K$, where $\wh U^\fra_K$ is the rigid analytic space over $K$ associated to the formal completion of $U^\fra\otimes\cO_K$ with respect to the special fiber.
The structure sheaf $\sO_{\wh U_K}$ on $\wh U_K$ is naturally an 
$F_+^\times$-equivariant sheaf in the sense of \S\ref{subsection: equivariant},
where $F_+^\times$ is the set of totally positive elements $F_+$ of $F$ viewed as a group with respect to the multiplication.

For any $k\in\bbZ$, we construct the 
\textit{$p$-adic polylogarithm} as a class
\[
	\Li^\p_{k}(t)\in H^{g-1}\bigl(\wh U_K/F_+^\times,\sO_{\wh U_K}\bigr),
\]
where $H^{g-1}\bigl(\wh U_K/F_+^\times,\sO_{\wh{U}_K}\bigr)$ is the equivariant cohomology of 
$\wh U_K$ with coefficients in $\sO_{\wh U_K}$. 
 %(see Theorem \ref{thm: polylog} and Definition \ref{def: k} for details).
For any nontrivial torsion point $\xi$ of $\wh U_K$, 
we denote by $\Delta_\xi$ the group of totally positive \emph{units} in $F^\times_+$ which preserve $\xi$. 
We then view $\xi$ as a rigid analytic space $\xi=\Sp K(\xi)$ with trivial action of $\Delta_\xi$. 
%we view $\xi$ as a scheme $\xi\coloneqq
%\Spec K(\xi)$ with trivial action of $\Delta_\xi$, where $\Delta_\xi$ is the multiplicative group 
%of totally positive \textit{units} in $F_+^\times$ which acts trivially on $\xi$.
We define the \textit{value} $\Li^{\p}_{k}(\xi)$ of the $p$-adic polylogarithm at $\xi$ 
to be the image of $\Li^{\p}_{k}(t)$ by the %composition of the
specialization map 
\[
    H^{g-1}\bigl(\wh U_K/F_+^\times,\sO_{\wh{U}_K}\bigr)\rightarrow
    H^{g-1}(\xi/\Delta_\xi,\sO_\xi)\cong K(\xi)
\]
induced from the equivariant morphism $\xi\rightarrow\wh{U}_K$.
%and the canonical isomorphism,
%$
%	H^{g-1}(\xi/\Delta_\xi,\sO_\xi)\cong K(\xi).
%$

For any nonzero integral ideal $\frg$ of $F$, we denote by $\Cl^+_F(\frg)$ 
the narrow ray class group modulo $\frg$ of $F$.
Let $\frg\neq(1)$. For any finite primitive Hecke character $\chi\colon\Cl^+_F(\frg)\rightarrow\bbC^\times$, 
Deligne-Ribet \cite{DR80}, Barsky \cite{Bar78} and Cassou-Nogu\`es  \cite{CN79}
defined the $p$-adic $L$-function $L_p(\chi,s)$ as an analytic function for $s\in\bbZ_p$
satisfying the interpolation property
\[
	L_p(\chi,-k)=\bigg(\prod_{\frp\divides(p)}\Big(1 - \chi\omega^{-k-1}_p(\frp)N\frp^{k}\Big)\bigg)L(\chi \omega_p^{-k-1},-k)
\]
for any integer $k\geq0$,
where $\omega_p$ denotes the composition of the norm map with the Teichm\"uller character.
Our main theorem, Theorem \ref{thm: main}
gives the following.

\begin{theorem}[=Corollary \ref{cor: main}]\label{thm: 2}
	Suppose $\frg$ does not divide any power of $(p)$,
	and let $\xi$ be an arbitrary primitive $\frg$-torsion point of $\bbT^\fra[\frg]$
	for some fractional ideal $\fra$ prime to $\frg$.
	Then for any integer $k\in\bbZ$, we have
	\[
		L_p(\chi\omega_p^{1-k},k)=\frac{g(\chi,\xi)}{N\frg}
		\sum_{\frb\in\Cl_F^+(\frg)}
		\chi(\frb)^{-1}\,\Li_k^\p(\xi^\frb),
	\]
	where $g(\chi,\xi)$ is a certain Gauss sum associated to the Hecke character $\chi$, and $\xi^\frb$ is a certain torsion point in $\bbT^{\fra\frb}[\frg]$.
\end{theorem}

We remark that
there is another approach to $L$-values and $p$-adic $L$-functions of totally real fields via cohomology classes, 
by Sczech \cite{MR1231838}, Solomon \cites{MR1631700,MR1670874}, Hu--Solomon \cite{MR1794257}, Hill \cite{MR2392823}, 
Spiess \cite{MR3201900}, Charollois--Dasgupta \cite{MR3272012} and Charollois--Dasgupta--Greenberg \cite{MR3351752}. 
They mainly study certain group cocycles on $\mathrm{GL}_n(\bbQ)$, 
while our method relies on the geometric objects $\bbT$ and $\sT$. 
It would be interesting to clarify how their method is related with the method of the current article. 
We also note that Bekki \cite{Bek21} studies $L$-values of general number fields 
by incorporating the ideas of both the above-mentioned works and this article.

Be\u{\i}linson and Deligne gave an interpretation of the complex polylogarithm functions \eqref{eq: polylog}
as periods of certain unipotent pro-variation of mixed $\bbR$-Hodge structures, referred to as the \textit{polylogarithm sheaf},
on the projective line minus three points.  Furthermore, this polylogarithm sheaf was interpreted as the Hodge
realization of the \textit{motivic polylogarithm class} associated with the algebraic torus $\bbG_m$, and the fact
that the regulator maps are expressed by \eqref{eq: polylog} follows from this fact (see for example \cite{HW98}).
The motivic polylogarithm class for general commutative group schemes were constructed by Huber and Kings \cite{HK18}.
The $p$-adic interpolation of critical values of Hecke $L$-functions of totally real fields was previously studied by Be\u{\i}linson, Kings, and Levin in \cite{BKL18} via the topological polylogarithm on a torus.
Our construction of the $p$-adic polylogarithm arose from our effort to understand the syntomic realization of the motivic polylogarithm class for $\bbT$,
in order to also capture the noncritical values of the $p$-adic Hecke $L$-functions.

In \cite{Gro90}, Gros gave a working definition of syntomic cohomology and the syntomic regulator map,
and Kurihara \cite{Gro90}*{Appendix} proved that the syntomic regulators in the cyclotomic case are
expressed by the $p$-adic polylogarithm functions \eqref{eq: classical polylogarithm}.  
The interpretation of these results in terms of syntomic realization of the polylogarithm
sheaf was given in \cite{Ban00-0}.
The relation of the $p$-adic polylogarithms in our case
to the syntomic realization of the polylogarithm classes,
as well as its plectic incarnation following the spirit of Nekov\'{a}\v{r} and Scholl 
\cite{NS16}, will be explored in subsequent research.

\tableofcontents

The content of this article is as follows.
In \S\ref{section: Shintani}, we give the definition of Lerch zeta functions and 
the construction of the Shintani generating class, generalizing 
the result of \cite{BHY19} to the case when the narrow class 
number of the totally real field is not necessarily one.
In \S\ref{section: polylogarithm},  we give  the construction of the $p$-adic
polylogarithm.
In \S\ref{section: interpolation}, we construct a $p$-adic measure interpolating
values at nonpositive integers of Lerch zeta functions,  and use this measure to  construct
the $p$-adic $L$-function $L_p(\chi,s)$.  Our main theorem, Theorem \ref{thm: main},
will be proved  in \S\ref{section: main theorem}.

%%%%%%%%%%%%%%%%%%%%%%%%%%%%%%%%%%%%%%%%%%%%%%%%%%
%
%
%
\section{Lerch Zeta Functions and the Shintani Generating Class}\label{section: Shintani}
%
%
%
%%%%%%%%%%%%%%%%%%%%%%%%%%%%%%%%%%%%%%%%%%%%%%%%%%

In this section, we generalize the construction of the Shintani Generating Class given
in \cite{BHY19} to deal with the case when the narrow class number of the totally real
field is not necessarily one.
In particular, we introduce the scheme
$\bbT\coloneqq \coprod_{\fra\in\frI}\bbT^\fra$, where $\frI$ denotes the set
of nonzero fractional ideals of the totally real field
and $\bbT^\fra\coloneqq\Hom_\bbZ(\fra,\bbG_m)$.

%%%%%%%%%%%%%%%%%%%%%%%%%%%%%%%%%%%%%%%%%%%%%%%%%%
%
\subsection{Lerch Zeta Functions and Hecke $L$-functions}\label{subsection: Lerch}
%
%%%%%%%%%%%%%%%%%%%%%%%%%%%%%%%%%%%%%%%%%%%%%%%%%%

Let $F$ be a totally real field of degree $g$ with ring of
integers $\cO_F$.  
In this subsection, we first give the definition of the Lerch zeta functions
of $F$. 
We then introduce the set $\sT_\tors$ as a natural parameter space for the Lerch zeta functions,
and express the $L$-function
associated to a Hecke character of $F$ canonically in terms of the Lerch zeta functions.

In what follows, for any subset $X$ of $F$, we denote by $X_+$ the set of totally positive elements of $X$.
For any fractional ideal $\fra$ of $F$, we let $\bbT^\fra$ be the algebraic torus
\[
	\bbT^\fra\coloneqq\Hom(\fra,\bbG_m)
\]
defined over $\bbZ$ used by Katz \cite{Katz81}.
The algebraic torus $\bbT^\fra$ represents the functor associating
to any $\bbZ$-algebra $R$ the group of additive characters
$\bbT^\fra(R)=\Hom(\fra,R^\times)$ on $\fra$ with values  in $R^\times$.
$\bbT^\fra$ is given as the affine scheme $\bbT^\fra=\Spec\bbZ[t^\alpha\mid\alpha\in\fra]$, 
where $t^\alpha$ are parameters satisfying $t^\alpha t^{\alpha'}=t^{\alpha+\alpha'}$ for any $\alpha,\alpha'\in\fra$.  
We let $\Delta\coloneqq(\cO_{F}^\times)_+$ be the group of totally positive units of $F$.
Then the natural left action of $\Delta$ on $\fra$ induces a right action of $\Delta$ on $\bbT^\fra$.
The action of $\e\in\Delta$, denoted by $\bra{\e}\colon\bbT^\fra\to\bbT^\fra$, maps 
a character $\xi\in\bbT^\fra$ to the character $\xi^\e$ defined by 
$\xi^\e(\alpha)\coloneqq\xi(\e\alpha)$. 
In terms of the coordinate ring, the isomorphism $\bra{\e}\colon\bbT^\fra\rightarrow\bbT^\fra$ 
is given by $t^\alpha\mapsto t^{\e\alpha}$ for any $\alpha\in\fra$.  

For any torsion point $\xi\in\bbT^\fra(\ol\bbQ)$, we define a function $\xi\Delta$ 
on $\fra$ to be the sum over the $\Delta$-orbit of $\xi$, i.e., 
\[
	\xiDelta\coloneqq\sum_{\e\in\Delta_\xi\backslash\Delta} \xi^\e,
\]
where $\Delta_\xi\coloneqq\{\e\in\Delta\mid\xi^\e=\xi\}\subset\Delta$ is 
the isotropic subgroup of $\xi$.
By definition, $\xiDelta$ satisfies $\xi^\e\!\Delta=\xiDelta$ for any
$\e\in\Delta$ and defines a map $\xiDelta\colon\Delta\backslash\fra\rightarrow\ol\bbQ$.
We define the Lerch zeta function as follows.

\begin{definition}
	Let $\fra$ be a nonzero 
	fractional ideal of $F$, and let $\xi$ be a torsion point of $\bbT^\fra(\ol\bbQ)$.  
	We define the \textit{Lerch zeta function} for $\xiDelta$ by the series
	\begin{equation}\label{eq: Lerch}
		\cL(\xiDelta,s)\coloneqq\sum_{\alpha\in\Delta\backslash\fra_+} \xiDelta(\alpha) N(\fra^{-1}\alpha)^{-s}
	\end{equation}
	for $s\in\bbC$ with $\Re(s)>1$.
\end{definition}

This is a natural generalization of the Lerch zeta function defined in \cite{BHY19}*{Definition 1.3}.
The sum converges absolutely for $\Re(s)>1$.  The function continues meromorphically 
to the whole complex plane, and is entire if $\xi\neq 1$.

We next introduce the set $\sT_\tors$
as a natural parameter space for the Lerch zeta functions. 
The action of $\Delta$ on $\bbT^\fra$ generalizes to isomorphisms given by 
elements of $F_+^\times$
as follows.  For any $x\in F_+^\times$, the multiplication by $x$ gives an isomorphism 
$\fra\cong x\fra$ of $O_F$-modules, and hence induces an 
isomorphism of group schemes
\begin{equation}\label{eq: ab}
	\bra{x}\colon\bbT^{x\fra}\to\bbT^{\fra}.
\end{equation}
Explicitly, this isomorphism maps any character $\xi\in\bbT^{x\fra}(R)$ 
to the character $\xi^x\in\bbT^{\fra}(R)$ given by $\xi^x(\alpha)\coloneqq\xi(x\alpha)$ 
for $\alpha\in\fra$. If $\frI$ denotes the group of nonzero fractional ideals of $F$,
then $\bbT\coloneqq\coprod_{\fra\in\frI} \bbT^\fra$ has a natural action of $F_+^\times$
given by the isomorphism
\begin{equation}\label{eq: action}
	\bra{x}\colon\bbT\rightarrow \bbT
\end{equation}
for any $x\in F_+^\times$
obtained as the collection of isomorphisms 
$\bra{x}\colon\bbT^{x\fra}\rightarrow\bbT^{\fra}$ for all $\fra\in\frI$. 

\begin{definition}\label{def: stack}
	We define $\sT(\ol\bbQ)$ to be the quotient set 
	\[
		\sT(\ol\bbQ)\coloneqq\bbT(\ol\bbQ)/F_+^\times,
	\]
	and let $\sT_\tors\subset \sT(\ol\bbQ)$ be the set of points in $\sT(\ol\bbQ)$
	represented by torsion points of 
	$\bbT^\fra(\overline{\bbQ})$ for $\fra\in\frI$.
Note that if we fix a set of fractional ideals $\frC$ of $F$ representing the narrow ideal 
class group $\Cl^+_F(1)\coloneqq\frI/P_+$ (where $P_+\coloneqq\{(x)\mid x\in F_+^\times\}$), 
then we have
$
	\sT(\ol\bbQ)=\coprod_{\fra\in\frC}(\bbT^\fra(\ol\bbQ)/\Delta).
$
\end{definition}

The following lemma shows that $\sT_\tors$ is the natural parameter space 
for the Lerch zeta functions. 

\begin{lemma}%\label{lem: one}
	Let $\fra$ be a nonzero fractional ideal of $F$ and let $x\in F_+^\times$. 
	Then for any torsion point $\xi\in\bbT^{x\fra}(\ol\bbQ)$, we have
	\[
		\cL(\xiDelta,s) = \cL(\xi^x\!\Delta,s),
	\]
	where $\xi^x$ is the torsion point of $\bbT^{\fra}(\ol\bbQ)$ corresponding to $\xi$ 
	through the isomorphism \eqref{eq: ab}. In other words, the Lerch zeta function 
	depends only on the class of $\xi$ in $\sT_\tors$.
\end{lemma}

\begin{proof}
	By definitions of the Lerch zeta functions and of $\xi^x$, we have
	\[
		\cL(\xiDelta,s)=\sum_{\beta\in\Delta\backslash x\fra_+} \xiDelta(\beta) 
		N(\fra^{-1}x^{-1}\beta)^{-s}
		=\sum_{\alpha\in\Delta\backslash\fra_+} \xiDelta(x\alpha) N(\fra^{-1}\alpha)^{-s}
		=\cL(\xi^x\!\Delta,s)
	\]
	as desired.
\end{proof}

We will show in Proposition \ref{prop: Hecke} that
the Lerch zeta functions may be used to express the Hecke $L$-functions of Hecke characters 
of $F$. We first review some results concerning the finite Fourier transform.
For a fractional ideal $\fra\in\frI$ and a nonzero integral ideal $\frg$, 
we denote by $\bbT^\fra[\frg]\coloneqq\Hom_\bbZ(\fra/\frg\fra,\ol\bbQ^\times)$ 
the character group of the finite abelian group $\fra/\frg\fra$. 
We call it the group of $\frg$-torsion points of $\bbT^\fra(\ol{\bbQ})$, 
since the $\cO_F$-action on $\fra$ induces an $\cO_F$-module structure 
on $\bbT^\fra(\ol{\bbQ})$, and its submodule of $\frg$-torsion points is identified 
with $\bbT^\fra[\frg]$ through the natural inclusion 
$\bbT^\fra[\frg]\hookrightarrow\bbT^\fra(\ol{\bbQ})$. 
Now, for a function $\phi\colon\fra/\frg\fra\rightarrow\bbC$ and $\xi\in\bbT^\fra[\frg]$, let
\[
	c_\phi(\xi)\coloneqq N\frg^{-1}\sum_{\beta\in\fra/\frg\fra} \phi(\beta)\xi(-\beta). 
\]
Then by the Fourier inversion formula for functions on $\fra/\frg\fra$, we have
\[
	\phi(\alpha)=\sum_{\xi\in\bbT^\fra[\frg]}c_\phi(\xi)\xi(\alpha)
\]
for any $\alpha\in\fra/\frg\fra$.  Moreover, if $\phi^\e=\phi$ for any $\e\in\Delta$,
where $\phi^\e$ is the function on $\fra/\frg\fra$ defined by $\phi^\e(\alpha)=\phi(\e\alpha)$, 
then we have $c_\phi(\xi^\e)=c_\phi(\xi)$, hence we see that
\begin{equation}\label{eq: equivariant Fourier}
	\phi(\alpha)=\sum_{\xi\in\bbT^\fra[\frg]/\Delta}c_\phi(\xi)\xiDelta(\alpha)
\end{equation}
for any $\alpha\in\fra/\frg\fra$ in this case.

In what follows, for any nonzero integral ideal  $\frg$ of $F$, 
we let $\Cl^+_F(\frg)$ be the narrow ray class group of $F$ modulo $\frg$. 
That is, we let $\Cl_F^+(\frg)\coloneqq \frI_\frg/P_+(\frg)$,
where $\frI_\frg$ is the group of nonzero fractional ideals of $F$ prime to $\frg$ and 
$P_+(\frg)\coloneqq\{(x)\mid x\in F_+^\times,\,x\equiv1\operatorname{mod}^\times\frg\}$. 
For $\fra\in\frI$, we denote by $(\fra/\frg\fra)^\times$ the subset of $\fra/\frg\fra$ 
consisting of all elements which generate $\fra/\frg\fra$ as an $\cO_F/\frg$-module. 
Then we have the following description of $\Cl^+_F(\frg)$.

\begin{lemma}\label{lem: ray class group}
For any $\fra\in\frI$, we have a well-defined map 
\[
    (\fra/\frg\fra)^\times\longrightarrow\Cl^+_F(\frg) 
\]
which sends a residue class of $\alpha\in\fra_+$ to the ray class of $\fra^{-1}\alpha$. 
Moreover, if $\frC\subset\frI$ is a set of representatives of 
the narrow ideal class group $\Cl^+_F(1)$, 
the above maps for $\fra\in\frC$ induce a bijection
\[
    \coprod_{\fra\in\frC}\Delta\backslash(\fra/\frg\fra)^\times\longrightarrow\Cl^+_F(\frg). 
\]
\end{lemma}

A finite Hecke character of $F$ is a homomorphism $\chi\colon\Cl^+_F(\frg)\rightarrow\bbC^\times$ 
for some integral ideal $\frg$. 
%Then by \cite{Neu99}*{Chapter VII, (6.9) Proposition}, 
%there exists a unique character $\chif\colon(\cO_F/\frg)^\times\rightarrow\bbC^\times$
%such that $\chi((\alpha))=\chif(\alpha)$ for any $\alpha\in \cO_{F+}$ prime to $\frg$.
By composing with the map $(\fra/\frg\fra)^\times\to\Cl^+_F(\frg)$ given in 
Lemma \ref{lem: ray class group}, we define a map 
$\chi_\fra\colon(\fra/\frg\fra)^\times\to\bbC^\times$ for each $\fra\in\frI$. 
Then we have 
\begin{equation}\label{eq: chi_fra}
    \chi_{\fra\frb}(\alpha\beta)=\chi_\fra(\alpha)\chi_\frb(\beta) 
\end{equation}
for any $\fra,\frb\in\frI$ and $\alpha\in(\fra/\frg\fra)^\times$, $\beta\in(\frb/\frg\frb)^\times$. 

By extension by \textit{zero}, we will often regard $\chi$ as a map $\chi:\frI\rightarrow\bbC$, 
and $\chi_\fra$ as a map $\chi_\fra\colon\fra/\frg\fra\rightarrow\bbC$. 
Then the above formula \eqref{eq: chi_fra} holds for any $\alpha\in\fra$ and $\beta\in\frb$. 

\begin{definition}%\label{def: finite fourier}
	Let $\chi\colon\Cl^+_F(\frg)\rightarrow\bbC^\times$ be a finite Hecke character.
	For any $\xi\in\bbT^\fra[\frg]$, we let
	\[
		c_\chi(\xi)\coloneqq c_{\chi_\fra}(\xi).
	\]
	Note that by definition, for any $\e\in\Delta$, we have $\chi_\fra^\e=\chi_\fra$,
	hence $c_{\chi}(\xi^\e)=c_{\chi}(\xi)$ for any $\xi\in\bbT^\fra[\frg]$.
\end{definition}

\begin{lemma}\label{lem: two}
	Let $\fra\in\frI$ and $x\in F_+^\times$. 
	Then for any torsion point $\xi\in\bbT^{x\fra}[\frg]$, we have
	\[
		c_{\chi}(\xi)=c_{\chi}(\xi^x),
	\]
	where $\xi^x$ is the torsion point of $\bbT^{\fra}[\frg]$ 
	corresponding to $\xi$ through the isomorphism \eqref{eq: ab}.
	In other words, the constant $c_\chi(\xi)$ depends only on the class of $\xi$ in $\sT_\tors$.
\end{lemma}

\begin{proof}
	We have
	\[
		c_{\chi}(\xi)=N\frg^{-1}\sum_{\beta\in x\fra/\frg x\fra}\chi_{x\fra}(\beta)\xi(-\beta)
		=N\frg^{-1}\sum_{\alpha\in\fra/\frg\fra}\chi_\fra(\alpha)\xi(-x\alpha)
		 = c_{\chi}(\xi^x)
	\]
	as desired.
\end{proof}

Let
$\sT[\frg]\coloneqq\bigl(\coprod_{\fra\in\frI}\bbT^\fra[\frg]\bigr)/F_+^\times$,
which is a finite set.
The Hecke $L$-function $L(\chi,s)$ associated to a finite Hecke character
$\chi\colon\Cl^+_F(\frg)\rightarrow\bbC^\times$ is defined by the series
\[
    L(\chi,s)=\sum_{\fra\subset\cO_F}\chi(\fra)N\fra^{-s},
\]
which is absolutely convergent for $\Re(s)>1$.
The Hecke $L$-function may be expressed in terms of the Lerch zeta functions as follows.

\begin{proposition}\label{prop: Hecke}
	Let $\chi\colon\Cl^+_F(\frg)\rightarrow\bbC^\times$ be a finite Hecke  character  and let  $L(\chi,s)$ be
	the Hecke $L$-function of $\chi$.
	Then we have
	\begin{equation}\label{eq: first eq}
		L(\chi,s)=\sum_{\xi\in\sT[\frg]}c_{\chi}(\xi)\cL(\xiDelta,s).
	\end{equation}
\end{proposition}

\begin{proof}
    In what follows, we let $\frC\subset\frI$ be a set of representatives of the group $\Cl^+_F(1)$.
	Then $\{\fra^{-1}\mid\fra\in\frC\}$ also represents the group $\Cl^+_F(1)$, 
	and the set of integral ideals of $F$ is given by 
	$\{ \fra^{-1}\alpha \mid \fra\in\frC, \alpha\in\fra_+ \}$.
	By definition of the Hecke $L$-function, we have
	\begin{align*}
		L(\chi,s)&=\sum_{\fra\subset\cO_F} \chi(\fra)N\fra^{-s}
		=\sum_{\fra\in\frC}\sum_{\alpha\in\Delta\backslash\fra_+}
		\chi(\fra^{-1}\alpha)N(\fra^{-1}\alpha)^{-s}
		=\sum_{\fra\in\frC} \sum_{\alpha\in\Delta\backslash\fra_+}
		\chi_\fra(\alpha)N(\fra^{-1}\alpha)^{-s}\\
		&
		=\sum_{\fra\in\frC}\sum_{\alpha\in\Delta\backslash\fra_+}
		\sum_{\xi\in\bbT^\fra[\frg]/\Delta}  c_{\chi}(\xi) \xiDelta(\alpha) N(\fra^{-1}\alpha)^{-s}
		=\sum_{\substack{\fra\in\frC\\\xi\in\bbT^\fra[\frg]/\Delta}} c_{\chi}(\xi)\cL(\xiDelta,s),
	\end{align*}
	where the fourth equality follows from \eqref{eq: equivariant Fourier} and 
	the last equality follows from the definition of the Lerch zeta function \eqref{eq: Lerch}.
	This proves our equality \eqref{eq: first eq}, since we have a natural bijection 
	$\sT[\frg]\cong\coprod_{\fra\in\frC}(\bbT^\fra[\frg]/\Delta)$.
\end{proof}

Recall that a Hecke character $\chi\colon\Cl^+_F(\frg)\rightarrow\bbC^\times$ is 
called \emph{primitive} of conductor $\frg$ if it does not factor through $\Cl^+_F(\frg')$ 
for any $\frg'\supsetneq\frg$. Similarly, we say that an additive character 
$\xi\in\bbT^\fra[\frg]$ is \emph{primitive} if it does not factor through $\fra/\frg'\fra$ 
for any $\frg'\supsetneq\frg$. 
When the Hecke character is primitive, we may restrict the sum in \eqref{eq: first eq} 
to the subset of primitive additive characters. 
We denote by $\bbT^\fra_\prim[\frg]$ 
the set of primitive elements in $\bbT^\fra[\frg]$, and set 
$\sT_\prim[\frg]\coloneqq\bigl(\coprod_{\fra\in\frI}\bbT^\fra_\prim[\frg]\bigr)/F_+^\times$. 

\begin{proposition}\label{prop: Hecke primitive}
	If a finite Hecke character $\chi\colon\Cl^+_F(\frg)\rightarrow\bbC^\times$ is primitive, 
	then $c_\chi(\xi)=0$ holds for any non-primitive $\xi\in\bbT^\fra[\frg]$.
	Thus we have 
	\begin{equation}\label{eq: Hecke-Lerch primitive}
		L(\chi,s)=\sum_{\xi\in\sT_\prim[\frg]}c_{\chi}(\xi)\cL(\xiDelta,s).
	\end{equation}
\end{proposition}

\begin{proof}
	Let $\xi\in\bbT^\fra[\frg']$ for some $\frg'\supsetneq\frg$. 
	Since $\chi$ does not factor through $\Cl^+_F(\frg')$, 
	there exists an element  $\gamma\in\cO_{F+}$ prime to $\frg$ such that
	$\gamma\equiv 1\pmod{\frg'}$ and $\chi_{\cO_F}(\gamma)\neq 1$.
	Then we have $\xi(\alpha)=\xi(\gamma\alpha)$ for any $\alpha\in\fra$, hence 
	\begin{align*}
		c_\chi(\xi)&=N\frg^{-1}\sum_{\alpha\in\fra/\frg\fra}\chi_\fra(\alpha)\xi(-\alpha)
		=N\frg^{-1}\sum_{\alpha\in\fra/\frg\fra}\chi_\fra(\gamma\alpha)\xi(-\gamma\alpha)\\
	    &=N\frg^{-1}\chi_{\cO_F}(\gamma)\sum_{\alpha\in\fra/\frg\fra}\chi_\fra(\alpha)\xi(-\alpha)
	    =\chi_{\cO_F}(\gamma)c_\chi(\xi).
	\end{align*}
	This shows that $c_\chi(\xi)=0$ as desired, and the identity 
	\eqref{eq: Hecke-Lerch primitive} follows immediately from \eqref{eq: first eq}. 
\end{proof}

\begin{remark}
	By \cite{stacks}*{Tag 04TK Theorem 17.3}, the quotient
	\[
		\sT\coloneqq \bbT/F_+^\times
	\] 
	of $\bbT\coloneqq\coprod_{\fra\in\frI}\bbT^\fra$ with respect to the action of $F_+^\times$
	is an algebraic stack.  Then the set $\sT(\ol\bbQ)$ of Definition \ref{def: stack} is the
	set of $\ol\bbQ$-valued points of $\sT$.
	By construction, we have an isomorphism
	$
		\sT \cong \coprod_{\fra\in\frC}(\bbT^\fra/\Delta)
	$
	as algebraic stacks,
	where $\frC$ is a set of fractional ideals of $F$ representing the classes of $\Cl_F^+(1)$.
\end{remark}

%%%%%%%%%%%%%%%%%%%%%%%%%%%%%%%%%%%%%%%%%%%%%%%%%%
%
\subsection{The Shintani Generating Class}\label{subsection: equivariant}
%
%%%%%%%%%%%%%%%%%%%%%%%%%%%%%%%%%%%%%%%%%%%%%%%%%%

In this section, we construct the Shintani generating class following \cite{BHY19}*{\S 4},
however as an equivariant cohomology class in $U\coloneqq\coprod_{\fra\in\frI} U^\fra$
with respect to the action of $F_+^\times$,
where $U^\fra\coloneqq\bbT^\fra\setminus\{1\}$ for any $\fra\in\frI$.

We first review the definition of equivariant cohomology.
Let $G$ be a group with identity $e$.  In what follows, we let $X$ be a \textit{$G$-scheme},
i.e., a scheme $X$ equipped with a right action of $G$.
We denote by  $[u]\colon X \rightarrow X$ the action of $u\in G$, so that $[uv] = [v]\circ[u]$ for any $u,v\in G$ holds.
A \textit{$G$-equivariant sheaf} $\sF$ on $X$ is an $\sO_X$-module with isomorphisms
$\iota_u\colon[u]^*\sF\cong\sF$ for any $u\in G$, such that
$\iota_e=\id$ and the diagram
\[\xymatrix{
		 [uv]^* \sF \ar[r]^-{\iota_{uv}}\ar@{=}[d] &  \sF  \\
		 [u]^*[v]^* \sF  \ar[r]^-{[u]^*\iota_v} &   [u]^*\sF \ar[u]_{\iota_u}
	}\]
is commutative for any $u, v\in G$.
The structure sheaf $\sO_X$ itself is naturally a $G$-equivariant sheaf. 
The category of $G$-equivariant sheaves has enough injectives 
by \cite{Gro57}*{Proposition 5.1.1 and Th\'{e}or\`{e}me 1.10.1}. 
We define the \textit{equivariant global section} of a $G$-equivariant sheaf $\sF$ by 
\[
	\Gamma(X/G,\sF)\coloneqq\Hom_{\bbZ[G]}(\bbZ,\Gamma(X,\sF))=\Gamma(X,\sF)^G.
\]

\begin{definition}
	For any integer $m$, we define the equivariant cohomology $H^m(X/G,\sF)$ 
	of a $G$-equivariant sheaf $\sF$ to be the group obtained by applying to $\sF$ the 
	$m$-th derived functor of $\Gamma(X/G,-)$. 
\end{definition}

Let $\pi\colon G\rightarrow H$ be a homomorphism between groups.
For a $G$-scheme $X$ and an $H$-scheme $Y$, we say that a morphism of schemes $f\colon X\rightarrow Y$
is \textit{equivariant}, if the actions of $G$ and $H$ are compatible with $f$ through $\pi$.
For such $f$,
if $\sF$ is an $H$-equivariant sheaf on $Y$, then $f^*\sF$ naturally has a structure of a $G$-equivariant
sheaf on $X$, and $f$ induces a pullback morphism on equivariant cohomology
\[
	f^*\colon H^m(Y/H,\sF)\rightarrow H^m(X/G, f^*\sF).
\]

We now consider the scheme $\bbT$ of \S \ref{subsection: Lerch}.
Again, we let $F$ be a totally real field of degree $g\coloneqq[F\colon\bbQ]$, 
and we denote by $F_+^\times$ the set of totally positive elements in $F^\times$.
For any fractional ideal $\fra\in\frI$, we let $\bbT^\fra\coloneqq\Hom_\bbZ(\fra,\bbG_m)$.
Then by the action given in \eqref{eq: action}, the scheme 
$\bbT\coloneqq\coprod_{\fra\in\frI}\bbT^\fra$ is an $F_+^\times$-scheme. 
For any $x\in F_+^\times$, the isomorphism $\bra{x}\colon\bbT\cong\bbT$ is the collection 
of isomorphisms $\bra{x}\colon\bbT^{x\fra}\cong\bbT^{\fra}$ induced from the 
isomorphism $\fra\cong x\fra$ given by the multiplication by $x$ for any $\fra\in\frI$.

\begin{remark}\label{rem: description}
	Let $\frC$ be a set of fractional ideals of $F$ representing the classes of $\Cl^+_F(1)$.
	Then the category of $F_+^\times$-equivariant sheaves on $\bbT$ is equivalent to the 
	category of $\Delta$-equivariant sheaves on $\coprod_{\fra\in\frC}\bbT^\fra$.
\end{remark}

Next for $\fra\in\frI$, let $U^\fra\coloneqq\bbT^\fra\setminus\{1\}$.
Then any $x\in F_+^\times$ induces an 
isomorphism $\bra{x}\colon U^{x\fra}\rightarrow U^{\fra}$, hence 
$U\coloneqq\coprod_{\fra\in\frI} U^\fra$ is also an $F_+^\times$-scheme.
%The sheaf $\sO_{\bbT}(\bsk)$ of Definition \ref{def: twist} induces an equivariant sheaf on $U$.
We will now construct the \textit{equivariant \v Cech complex},
which may be used to express the equivariant cohomology of
$U$ with coefficients in an equivariant sheaf $\sF$ on $U$.

\begin{definition}%\label{def: A}
	For any fractional ideal $\fra$ in $\frI$, 
	we say that $\alpha\in\fra_+$ is \textit{primitive}, if $\alpha/N\not\in\fra_+$ for any integer $N>1$.
	We let $\sA_\fra\subset\fra_+$ be the set of primitive elements of $\fra_+$.
\end{definition}

If we let $U^\fra_\alpha\coloneqq\bbT^\fra\setminus\{t^\alpha=1\}$ for any $\alpha\in\sA_\fra$,
then $\frU^\fra\coloneqq\{U^\fra_\alpha\}_{\alpha\in \sA_\fra}$ gives an affine open covering of $U^\fra$ 
with a natural action of $\Delta$, and 
$\frU\coloneqq\{ U^\fra_\alpha\}_{\fra\in\frI,\alpha\in\sA_\fra}$ gives an affine open covering of 
$U=\coprod_{\fra\in\frI}U^\fra$ with a natural action on $F_+^\times$.
Let $q$ be an integer $\geq0$. For any $\bsalpha=(\alpha_0,\ldots,\alpha_q)\in\sA_\fra^{q+1}$, we let 
$U^\fra_\bsalpha\coloneqq U^\fra_{\alpha_0}\cap\cdots\cap U^\fra_{\alpha_q}$.
For an $F^\times_+$-equivariant sheaf $\sF$ on $U$, we write $\sF_\fra$ for the
restriction of $\sF$ to $U^\fra$, and denote by 
\begin{equation*}%\label{eq: alternating}
	\prod_{\bsalpha\in\sA_\fra^{q+1}}^\alt \Gamma(U^\fra_\bsalpha,\sF_\fra)
\end{equation*}
the subgroup of $ \prod_{\bsalpha\in\sA_\fra^{q+1}} \Gamma(U^\fra_\bsalpha,\sF_\fra)$ 
consisting of collections $\bss=(s_\bsalpha)$ of sections 
such that $s_{\rho(\bsalpha)} = \sgn(\rho)s_\bsalpha$ for any $\rho\in\frS_{q+1}$ and $s_\bsalpha=0$ if
$\alpha_i=\alpha_j$ for some $i\neq j$.

For any $x\in F_+^\times$ and $\bsalpha=(\alpha_0,\ldots,\alpha_q)\in\sA^{q+1}_\fra$, 
let $x\bsalpha=(x\alpha_0,\ldots,x\alpha_q)$. 
Then $\bra{x}\colon U^{x\fra}\isomto U^\fra$ induces 
$U^{x\fra}_{x\bsalpha}\isomto U^\fra_{\bsalpha}$, and 
we have isomorphisms 
\[
	\Gamma(U^\fra_{\bsalpha},\sF_{\fra}) \overset{\bra{x}^{*}}{\cong}\Gamma(U^{x\fra}_{x\bsalpha},
	\bra{x}^{*}\sF_{\fra})
	\overset{\iota^\fra_x}{\cong}\Gamma(U^{x\fra}_{x\bsalpha},\sF_{x\fra}),
\] 
which define a natural action of the group $F_+^\times$ on 
$\prod_{\fra\in\frI}\prod_{\bsalpha\in\sA_\fra^{q+1}}^\alt \Gamma(U^{\fra}_\bsalpha,\sF_\fra)$.
The \v{C}ech complex $C^\bullet(\frU/F_+^\times,\sF)$ is defined by
\[
	C^q(\frU/F_+^\times,\sF)\coloneqq
	\Biggl(\prod_{\fra\in\frI}
	\prod_{\bsalpha\in\sA_\fra^{q+1}}^\alt \Gamma(U^\fra_\bsalpha,\sF_\fra)\Biggr)^{F_+^\times}
\]
for each integer $q\geq 0$, 
with the differential $d^q\colon C^q(\frU/F_+^\times,\sF)\rightarrow 
C^{q+1}(\frU/F_+^\times,\sF)$ given by the alternating sum
\begin{equation}\label{eq: standard}
	(d^qf)_{\alpha_0\cdots\alpha_{q+1}}\coloneqq\sum_{j=0}^{q+1}(-1)^j
	f_{\alpha_0\cdots\breve\alpha_j\cdots\alpha_{q+1}}\big|_{U_{(\alpha_0,\cdots,\alpha_{q+1})}}.
\end{equation}
Then we have the following.

\begin{proposition}\label{prop: standard}
	Let $\sF$ be an $F^\times_+$-equivariant sheaf on $U$, and 
	assume that $\sF$ is quasi-coherent as an $\sO_U$-module. 
	Then for any integer $m\geq0$, the equivariant cohomology $H^m(U/F_+^\times,\sF)$ is given as
	\[
		H^m(U/F_+^\times,\sF)=H^m(C^\bullet(\frU/F_+^\times,\sF)).
	\]
\end{proposition}

\begin{proof}
This can be shown in the same way as \cite{BHY19}*{Corollary 3.6}. 
For the reader's convenience and for a later reference (\S\ref{subsection: p-adic polylog}), 
we briefly recall the argument. 

We define a complex of $F^\times_+$-equivariant sheaves $\sC^\bullet(\frU,\sF)$ on $U$ by 
\[\sC^q(\frU,\sF)\coloneqq \prod_{\fra\in\frI}\prod_{\bsalpha\in\sA_\fra^{q+1}}^\alt j_{\bsalpha*}j_{\bsalpha}^*\sF, \]
where $j_\bsalpha$ denotes the open immersion $U^\fra_\bsalpha\to U$. 
Then we have an exact sequence $0\to \sF\to \sC^\bullet(\frU,\sF)$ of equivariant sheaves. 
Since $C^\bullet(\frU/F_+^\times,\sF)=\Gamma(U/F^\times_+,\sC^\bullet(\frU,\sF))$, 
it suffices to show that each $\sC^q(\frU,\sF)$ is acyclic with respect to $\Gamma(U/F^\times_+,-)$. 

First let us observe that $\sC^q(\frU,\sF)$ is acyclic with respect to $\Gamma(U,-)$. 
Indeed, for any $m>0$, we have
\[
H^m(U, \sC^q(\frU, \sF)) 
\cong \prod_{\fra\in\frI}\prod_{\bsalpha\in\sA_\fra^{q+1}}^\alt H^m(U, j_{\bsalpha*}j_{\bsalpha}^*\sF)
\cong \prod_{\fra\in\frI}\prod_{\bsalpha\in\sA_\fra^{q+1}}^\alt H^m(U^\fra_\bsalpha, j_{\bsalpha}^*\sF)
=0. 
\]
Here we used the fact that $U^\fra_\bsalpha$ is affine, and the commutativity of the cohomology and the product: 
\begin{equation}\label{eq: cohomology and product}
H^m\biggl(U,\prod_{\lambda\in\Lambda}\sF_\lambda\biggr)\cong \prod_{\lambda\in\Lambda}H^m(U,\sF_\lambda), 
\end{equation}
where $(\sF_\lambda)_{\lambda\in\Lambda}$ is a family of quasi-coherent $\sO_U$-modules \cite{BHY19}*{Lemma 3.5} 
(note that the product $\prod_\lambda \sF_\lambda$ is taken not in the category of quasi-coherent $\sO_U$-modules 
but that of $\sO_U$-modules, though each $\sF_\lambda$ is quasi-coherent). 

Next we note that $H^a(F^\times_+,H^0(U,\sC^q(\frU,\sF)))=0$ for $a>0$. 
Indeed, we have 
\[H^0(U,\sC^q(\frU,\sF))=\prod_{\fra\in\frI}
\prod_{\bsalpha\in\sA_\fra^{q+1}}^\alt \Gamma(U^\fra_\bsalpha,\sF_\fra)
\cong \prod_{\fra\in\frI}
\prod_{\substack{\bsalpha=(\alpha_0,\ldots,\alpha_q)\in\sA_\fra^{q+1}\\
\alpha_0<\cdots<\alpha_q}} \Gamma(U^\fra_\bsalpha,\sF_\fra),\]
where $<$ is a fixed ordering on the set $\sA_\fra$ which is invariant under the $F^\times_+$-action 
(e.g., the order with respect to an embedding $F\to\bbR$). 
Then we see that this is a coinduced $F^\times_+$-module since the action of $F^\times_+$ on the set 
$\{(\fra,\bsalpha)\mid \fra\in\frI,\,\bsalpha=(\alpha_0,\ldots,\alpha_q)\in\sA_\fra^{q+1},\,\alpha_0<\cdots<\alpha_q\}$
is free. 

Finally, we use the spectral sequence 
\[E_2^{a,b}=H^a(F^\times_+,H^b(U,\sC^q(\frU,\sF)))\Longrightarrow H^{a+b}(U/F^\times_+,\sC^q(\frU,\sF)) \]
to see the desired acyclicity of $\sC^q(\frU,\sF)$. 
The existence of this spectral sequence follows from the fact that the functor 
\[\Gamma(U,-)\colon (\text{$F^\times_+$-equivariant sheaves on $U$})\to (\text{$F^\times_+$-modules})\] 
sends injective objects to acyclic ones with respect to $\Gamma(F^\times_+,-)$ 
(see \cite{Gro57}*{Lemme 5.6.1}). 
\end{proof}

Now we proceed to the construction of the Shintani generating class.
Note that we have a canonical isomorphism
\[
	F\otimes\bbR\cong\bbR^I\coloneqq \prod_{\tau\in I}\bbR,  \qquad 
	\alpha\otimes 1 \mapsto (\alpha^\tau),
\]
where $I=\Hom(F,\ol\bbQ)$
and we let $\alpha^\tau\coloneqq\tau(\alpha)$ for any embedding $\tau\in I$.
We denote by $\bbR^I_+\coloneqq\prod_{\tau\in I}\bbR_+$
the set of totally positive elements of $\bbR^I$, where $\bbR_+$ is the set of positive real numbers. 

\begin{definition}%\label{def: cone}
	An \textit{$F$-rational closed polyhedral cone} in $\bbR^I_+\cup\{0\}$, 
	which we also call simply a \textit{cone}, is a set of the form
	\[
		\sigma_\bsalpha\coloneqq\{  x_1 \alpha_1+\cdots+x_m\alpha_m \mid x_1,\ldots,x_m \in\bbR_{\geq0}\}
	\]
	for some $m\geq 0$ and $\bsalpha=(\alpha_1,\ldots,\alpha_m)\in F_+^m$. 
	In this case, we say that  $\bsalpha$ is a generator of $\sigma_\bsalpha$. 
	A cone is said to be \emph{simplicial},
	if it has a generator linearly independent over $\bbR$. 
\end{definition}

%By considering the case $m=0$, we see that $\sigma=\{0\}$ is a cone. 
We define the dimension $\dim\sigma$ of a cone $\sigma$ to be the dimension 
of the $\bbR$-vector space generated by $\sigma$. 
Note that, when a fractional ideal $\fra$ is given, 
any $m$-dimensional simplicial cone has a generator belonging to $\sA_\fra^m$, 
unique up to permutation. 

In what follows, we fix a numbering $I=\{\tau_1,\ldots,\tau_g\}$ of elements in $I$.
For any subset $R\subset\bbR_+^I\cup\{0\}$, we let
\[
	\breve R\coloneqq\{(u_{\tau_1},\ldots,u_{\tau_g})\in\bbR_+^I \mid\exists\delta>0,\,0<\forall\delta'<\delta,(u_{\tau_1},\ldots,u_{\tau_{g-1}},
	u_{\tau_g}-\delta')\in R\}.
\]
Then, for a cone $\sigma$, $\breve{\sigma}$ is nonempty if and only if $\sigma$ is $g$-dimensional
(cf.~\cite{Yam10}*{Lemma 5.3}). 

We define the function $\cG^\fra_\sigma(t)$ following \cite{BHY19}*{Lemma 4.1} as follows.

\begin{definition}\label{def: G}
	Let $\fra\in\frI$ and $\bsalpha=(\alpha_1,\ldots,\alpha_g)\in\sA_\fra^g$, 
	and put $\sigma=\sigma_\bsalpha$. Then we define 
	\[
		\cG^\fra_\sigma(t)\coloneqq 
		\frac{\sum_{\alpha\in\breve P_\bsalpha\cap\fra} t^\alpha}{(1-t^{\alpha_1})\cdots(1-t^{\alpha_g})}
	    \in\Gamma(U^\fra_\bsalpha,\sO_{\bbT}),
	\]
	where $P_\bsalpha\coloneqq\{ x_1\alpha_1+\cdots+x_g\alpha_g\mid
	\forall i\,\,0\leq x_i < 1\}$ is the parallelepiped spanned by $\alpha_1,\ldots,\alpha_g$.	
\end{definition}

The notation $\cG^\fra_\sigma(t)$ is justified by the fact that 
the generator $\bsalpha\in\sA_\fra^g$ is determined up to permutation by $\fra$ and $\sigma$ 
if $\dim\sigma=g$, while $\cG^\fra_\sigma(t)$ is zero if $\dim\sigma<g$ 
since $\breve{P}_\bsalpha=\emptyset$ in this case. 
We may also use the formal expansion 
\[	
	\cG^\fra_\sigma(t)=\sum_{\alpha\in\breve\sigma\cap\fra}t^\alpha, 
\]
which depends only on $\fra$ and $\sigma$. 
%hence $\cG_\sigma(t)$ depends only on the cone $\sigma$ and does not depend on the choice of the basis $\bsalpha$.
For any $\bsalpha=(\alpha_1,\ldots,\alpha_g)\in\sA_\fra^g$, 
let $\bigl(\alpha_j^{\tau_i}\bigr)$ be the matrix in $M_g(\bbR)$ whose $(i,j)$-component is $\alpha_j^{\tau_i}$.
We let $\sgn(\bsalpha)\in\{0,\pm1\}$ be the signature of  $\det\bigl(\alpha_j^{\tau_i}\bigr)$.
The Shintani generating class is constructed as in \cite{BHY19}*{Proposition 4.2} as follows.

\begin{proposition}%\label{prop: generating}
	For any $\fra\in\frI$ and $\bsalpha\in\sA_\fra^g$, let
	\[
		\cG^\fra_\bsalpha
		\coloneqq\sgn(\bsalpha)\cG^\fra_{\sigma_\bsalpha}(t)\in\Gamma(U^\fra_\bsalpha,\sO_{\bbT}).
	\]
	Then $(\cG^\fra_\bsalpha)$ is a cocycle in $C^{g-1}(\frU/F_+^\times,\sO_{\bbT})$, 
	hence defines a cohomology class
	\[
		\cG\coloneqq[\cG^\fra_\bsalpha]\in H^{g-1}(U/F_+^\times,\sO_{\bbT}),
	\]
	which we call the Shintani generating class.
\end{proposition}

\begin{proof}
	The fact that the collection $(\cG^\fra_\bsalpha)$ is invariant 
	under the action of $F_+^\times$ follows from the construction.
	The cocycle condition is proved using the same argument as that of \cite{BHY19}*{Proposition 4.2},
	again using the formula of \cite{Yam10}*{Proposition 6.2}.  (See the proof of 
	Theorem \ref{thm: polylog} below for a similar argument).
\end{proof}

When $F=\bbQ$, the Shintani generating class $\cG$ is simply the class represented by the 
rational function
\[
	\cG(t)=\frac{t}{1-t} \in \Gamma(U,\sO_{\bbG_m})=H^0(U, \sO_{\bbG_m})
\]
on $U=\bbG_m\setminus\{1\}$.

%%%%%%%%%%%%%%%%%%%%%%%%%%%%%%%%%%%%%%%%%%%%%%%%%%
%
%
%
\section{Construction of the $p$-adic Polylogarithms}\label{section: polylogarithm}
%
%
%
%%%%%%%%%%%%%%%%%%%%%%%%%%%%%%%%%%%%%%%%%%%%%%%%%%

Let $p$ be a rational prime.  
In this section, we will define the $p$-adic polylogarithm as a class in equivariant cohomology
of a certain rigid analytic space. 
In what follows, we fix a finite extension $K$ of $\bbQ_p$ in $\bbC_p$ containing 
$F^\tau\coloneqq\tau(F)$ for all $\tau\in I$.  
For any ring $R$, we denote by $\wh R:=\varprojlim_m R/p^mR$ the $p$-adic completion of $R$.

%%%%%%%%%%%%%%%%%%%%%%%%%%%%%%%%%%%%%%%%%%%%%%%%%%
%
\subsection{The $p$-adic Polylogarithm Function}%
%
%%%%%%%%%%%%%%%%%%%%%%%%%%%%%%%%%%%%%%%%%%%%%%%%%%

Let $\fra\in\frI$. We denote by $\wh\bbT^\fra_K$ the rigid analytic space over $K$ 
associated to the formal completion $\wh\bbT^\fra_{\cO_K}$ of 
$\bbT^\fra_{\cO_K}\coloneqq\bbT^\fra\otimes\cO_K$
with respect to the special fiber.
If we let $A\coloneqq\cO_K[t^\alpha\mid \alpha\in \fra]$, then
since $\bbT^\fra_{\cO_K}=\Spec A$, we have $\wh\bbT^\fra_{\cO_K}=\Spf\wh A$ 
and $\wh\bbT^\fra_K=\Sp(\wh A\otimes K)$ (see for example \cite{Bos14}*{\S7.4 Proposition 3}). 
Similarly, for any $\alpha\in\sA_\fra$, we let $\wh U^\fra_{\alpha\cO_K}$ be the 
formal scheme over $\cO_K$  obtained as the
$p$-adic completion of the affine scheme $U^\fra_\alpha\coloneqq\bbT^\fra\setminus\{t^\alpha=1\}$,
and $\wh U^\fra_{\alpha K}$ the associated rigid analytic space over $K$. 
Then we define a rigid analytic subspace $\wh U^\fra_K$ of $\wh\bbT^\fra_K$ by  
\[
	\wh U^\fra_K\coloneqq \bigcup_{\alpha\in\sA_\fra}\wh U^\fra_{\alpha K}.
\]
This $\wh U^\fra_K$ coincides with $\wh\bbT^\fra_K$ minus 
the residue disc around the identity, and the set 
$\wh \frU^\fra_K\coloneqq\bigl\{\wh U^\fra_{\alpha K}\bigr\}_{\alpha\in\sA_\fra}$ 
gives an affinoid open covering of $\wh U^\fra_K$.

Any torsion point $\xi\in\bbT^\fra(\ol{\bbQ})$ canonically defines 
a $K(\xi)$-valued point $\xi'\colon\Sp K(\xi)\rightarrow\wh\bbT^\fra_K$ as rigid analytic spaces.
In what follows, for any rigid analytic subspace $U\subset\wh\bbT_K$, we say that $\xi$ lies in $U$ if $\xi'$ does.
Abusing notation, we often denote $\xi'$ by $\xi$.
Note that $\xi$ lies in $\wh{U}^\fra_{\alpha K}$ if and only if $\xi(\alpha)\not\equiv 1$ modulo 
the maximal ideal of $O_{K(\xi)}$. In particular, any torsion point $\xi\in\bbT^\fra(\ol{\bbQ})$ 
which is not of $p$-power order lies in $\wh U^\fra_K$. 

For any $\fra\in\frI$, we denote by $(\fra\otimes\bbZ_p)^\times$ 
the set of generators of the $\cO_F\otimes\bbZ_p$-module $\fra\otimes\bbZ_p$, 
i.e., the set of $x\in\fra\otimes\bbZ_p$ such that $(\cO_F\otimes\bbZ_p)x=\fra\otimes\bbZ_p$. 
Equivalently, we may write it as 
\[
    (\fra\otimes\bbZ_p)^\times=\varprojlim_m (\fra/p^m\fra)^\times. 
\]
In particular, for $\alpha\in\fra$, we have $\alpha\otimes 1\in(\fra\otimes\bbZ_p)^\times$ 
if and only if the $\frp$-adic valuation of $\alpha$ is equal to that of $\fra$ 
for every prime $\frp\divides(p)$. In the following, we write simply $\alpha$ for 
$\alpha\otimes 1\in\fra\otimes\bbZ_p$. 

%If we take $u\in F^\times_+$ such that $u\fra$ is an integral ideal of $\cO_F$ prime to $(p)$, 
%then $\alpha\in(\fra\otimes\bbZ_p)^\times$ if and only if $u\alpha\in(\cO_F\otimes\bbZ_p)^\times$, 
%the latter denoting the set of invertible elements of the ring $\cO_F\otimes\bbZ_p$. 

%For any fractional ideal $\fra\in\frI$, there exists $u\in F^\times_+$ such that
%$u\fra$ is an integral ideal of $\cO_F$ prime to $(p)$. 
%We denote by $(\fra\otimes\bbZ_p)^\times$ the set of elements $\alpha\in \fra\otimes\bbZ_p$ 
%such that  $u\alpha\in (\cO_F\otimes\bbZ_p)^\times$ with respect to the inclusion $u\fra\subset\cO_F$. 
%Then the set $(\fra\otimes\bbZ_p)^\times$ is independent of the choice of $u\in F^\times_+$.

We define the $p$-adic polylogarithm function first as a formal power series as follows.

\begin{definition}\label{def: polylogarithm}
	Let $\fra\in\frI$.
	For $\bsalpha\in\sA_\fra^g$, $\sigma=\sigma_\bsalpha$
	and $\bsk\in\bbZ^I$,
	we define the \textit{$p$-adic polylogarithm function} 
	$\Li^{\fra,\p}_{\bsk,\sigma}(t)$ by the formal power series
	\begin{equation}\label{eq: polylogarithm}
		\Li^{\fra,\p}_{\bsk,\sigma}(t)
		\coloneqq\sum_{\substack{\alpha\in\breve\sigma\cap\fra\\ 
		\alpha\in(\fra\otimes\bbZ_p)^\times}} 
		\alpha^{-\bsk} t^\alpha,
	\end{equation}
	where $\alpha^{-\bsk}\coloneqq\prod_{\tau\in I}(\alpha^\tau)^{-k_\tau}$. 
\end{definition}

In the following, we will denote $\Li^{\fra,\p}_{\bsk,\sigma}(t)$ by $\Li^\p_{\bsk,\sigma}(t)$ 
for simplicity. 

%We note that the multi-indeu $\bsk$ fits neatly with the philosophy of plectic structures of  \cite{NS16}.
%by Nekov\'{a}\v{r} and Scholl

\begin{remark}
	\begin{enumerate}
		\item If $F=\bbQ$ and $\fra=\bbZ$, then the only element of $\sA_\fra$ is $1$, 
		which spans the cone $\sigma=\bbR_{\geq0}$. 
		In this case, 
		for any integer $k\in\bbN$, the function
		\[
			\Li^\p_{k,\sigma}(t) =
			\sum_{\substack{\alpha\in\bbN_+\\ 
			\alpha\in\bbZ_p^\times}} \alpha^{-k} t^\alpha
			=\sum_{\substack{n=1\\(n,p)=1}}^\infty n^{-k}t^n
		\]
		is the $p$-adic polylogarithm function $\ell^\p(t)$ of \cite{Del89}*{(3.2.2)},
		which coincides with the function \eqref{eq: classical polylogarithm} of 
		\S\ref{section: introduction}.
		\item If $\dim\sigma<g$, then we have $\Li^\p_{\bsk,\sigma}(t)=0$ since $\breve\sigma=\emptyset$.
		\item For any $g$-dimensional simplicial cone $\sigma$, $\bsalpha\in\sA^g_\fra$ such that $\sigma=\sigma_\bsalpha$ is determined uniquely up to exchange of components. 
	\end{enumerate}
\end{remark}

The power series $\Li^\p_{\bsk,\sigma}(t)$ of Definition \ref{def: polylogarithm} 
belongs, \emph{a priori}, to the completion $K\llbracket t^\alpha\mid\alpha\in\fra_+\rrbracket$ 
of the semigroup ring $K[t^\alpha\mid\alpha\in\fra_+]$ 
with respect to the ideal generated by $t^\alpha$ for all $\alpha\in\fra_+$. 
In fact, it lives in a smaller ring given below. 

\begin{lemma}\label{lem: polylogarithm}
    For $\fra\in\frI$ and $\bsalpha=(\alpha_1,\ldots,\alpha_g)\in\sA_\fra^g$, 
    we set $B\coloneqq\cO_K[t^\alpha\mid\alpha\in\fra_+]$ and 
    \begin{equation}\label{eq: B alpha}
        B_\bsalpha\coloneqq
		B\biggl[\frac{1}{1-t^{\alpha_1}},\cdots,\frac{1}{1-t^{\alpha_g}}\biggr]. 
    \end{equation}
    Then, for $\sigma=\sigma_\bsalpha$, we have $\Li^\p_{\bsk,\sigma}(t)\in\wh{B}_\bsalpha\otimes K$. 
    More precisely, $\Li^\p_{\bsk,\sigma}(t)$ belongs to the image of the natural inclusion 
    \begin{equation}\label{eq: wh B alpha K}
        \wh{B}_\bsalpha\otimes K\hookrightarrow K\llbracket t^\alpha\mid\alpha\in\fra_+\rrbracket. 
    \end{equation}
\end{lemma}

\begin{proof}
	It is sufficient to prove that, for a fixed $u\in(\fra^{-1}\otimes\bbZ_p)^\times$, 
	$u^{-\bsk}\Li^\p_{\bsk,\sigma}(t)$ belongs to the image of $\wh{B}_\bsalpha$. 
	We denote by $P_\bsalpha$ the parallelepiped generated by $\bsalpha$, and 
	let $R_m\coloneqq p^m \breve{P}_\bsalpha\cap\fra$ for any integer $m\geq 0$.
	Then we put 
	\[
		f_m(t)\coloneqq\frac{1}{\prod_{i=1}^g(1-t^{\alpha_ip^m})}
		\sum_{\substack{\alpha\in R_m\\ \alpha\in(\fra\otimes\bbZ_p)^\times}}
		(u\alpha)^{-\bsk}t^\alpha.
	\]
	In the sum, since $(u\alpha)^\tau\in\cO_K^\times$ for any $\tau\in I$, 
	we have $(u\alpha)^{-\bsk}\in\cO^\times_K$, hence $f_m(t) \in\wh B_\bsalpha$.
	Since 
	\[
		\breve\sigma\cap\fra = \coprod_{n_1,\ldots,n_g\in\bbN}
		 (R_m+\alpha_1p^mn_1+\ldots+\alpha_gp^mn_g)
	\]
	and
	\[
		\frac{1}{\prod_{i=1}^g(1-t^{\alpha_ip^m})}=\sum_{n_1,\ldots,n_g\in\bbN}
		t^{\alpha_1p^mn_1+\cdots+\alpha_gp^mn_g},
	\]
	we see from the definition of the formal power series $\Li^\p_{\bsk,\sigma}(t)$ that 
	\[
		u^{-\bsk}\Li^\p_{\bsk,\sigma}(t)
		=\sum_{\substack{\alpha\in\breve\sigma\cap\fra\\ \alpha\in(\fra\otimes\bbZ_p)^\times}} 
		(u\alpha)^{-\bsk} t^\alpha
		\equiv f_m(t)\pmod{p^m\cO_K\llbracket t^\alpha\mid\alpha\in\fra_+\rrbracket}.
	\]
	In particular, we have $f_m(t)\equiv f_n(t)\pmod{p^m\wh{B}_\bsalpha}$ for any integer $n\geq m$, 
	hence the sequence $\bigl(f_m(t)\bigr)_{m\in\bbN}$ is a Cauchy sequence in 
	$\wh{B}_\bsalpha$ for the $p$-adic topology. Thus the $p$-adic limit
	\[
		u^{-\bsk}\Li^\p_{\bsk,\sigma}(t)=\lim_{m\to\infty} f_m(t)
	\]
	gives an element in $\wh B_\bsalpha$, as desired. 
\end{proof}

In what follows, for any $\bsalpha=(\alpha_0,\ldots,\alpha_q)\in\sA^{q+1}_\fra$, let 
$\wh{U}^\fra_{\bsalpha K}\coloneqq\wh{U}^\fra_{\alpha_0K}\cap\cdots\cap \wh{U}^\fra_{\alpha_qK}$. 
Noting that $A=O_K[t^\alpha\mid\alpha\in\fra]$, if we let
\[
    A_\bsalpha\coloneqq A\biggl[\frac{1}{1-t^{\alpha_0}},\cdots,\frac{1}{1-t^{\alpha_q}}\biggr],
\]
then we have $\wh{U}^\fra_{\bsalpha K}=\Sp(\wh{A}_\bsalpha\otimes K)$.

\begin{proposition}\label{prop: continuation}
    For $\fra\in\frI$, $\bsalpha\in\sA_\fra^g$ and $\sigma=\sigma_\bsalpha$, 
    the $p$-adic polylogarithm function $\Li^\p_{\bsk,\sigma}(t)$ defines 
    a rigid analytic function on $\wh{U}^\fra_{\bsalpha K}$. 
\end{proposition}

\begin{proof}
    It follows from Lemma \ref{lem: polylogarithm}, 
    since we have a natural inclusion $\wh B_\bsalpha\subset\wh A_\bsalpha$.
\end{proof}

We will next consider the differential equation satisfied by the functions $\Li^\p_{\bsk,\sigma}(t)$.
For any $\tau\in I$ and $\fra\in\frI$, we let $\partial_\tau $ be the differential operator
\begin{equation}\label{eq: differential operator}
	\partial_\tau \coloneqq\sum_{j=1}^g  \gamma_j^\tau  \partial_{\log\gamma_j}
\end{equation}
on $\bbT^\fra$ for a choice of $\bbZ$-basis $\gamma_1,\ldots,\gamma_g$ of $\fra$, 
where $ \partial_{\log\gamma}\coloneqq t^\gamma\frac{\partial}{\partial t^\gamma}$ for any $\gamma\in\fra$.

\begin{lemma}\label{lem: differential}
	The differential operators $\partial_{\tau}$ act on any $t^\alpha$ for $\alpha\in\fra$ 
	by $\partial_{\tau}t^\alpha = \alpha^\tau t^\alpha$.
	In particular, $\partial_\tau$ is independent of the choice of the basis $\gamma_1,\ldots,\gamma_g$.
\end{lemma}

\begin{proof}
	For a $\bbZ$-basis $\gamma_1,\ldots,\gamma_g$ of $\fra$, we have 
	$\partial_\tau t^{\gamma_j} = \gamma^\tau_j t^{\gamma_j}$.
	For any $\alpha\in\fra$, if we write 
	$\alpha = m_1\gamma_1+\cdots+m_g\gamma_g$ for $\bsm=(m_j)\in\bbZ^g$,
	then we have
	\[
		\partial_\tau t^\alpha = \partial_\tau \bigl( (t^{\gamma_1})^{m_1}\cdots (t^{\gamma_g})^{m_g}\bigr)
		= \sum_{j=1}^g m_j\gamma_j^\tau\bigl( (t^{\gamma_1})^{m_1}\cdots (t^{\gamma_g})^{m_g}\bigr)
		=\alpha^\tau t^\alpha
	\]
	as desired.
\end{proof}

In fact, $\partial_\tau$ corresponds to the differential operator $\frac{1}{2\pi i}\frac{\partial}{\partial z_\tau}$
through the uniformization $(F\otimes\bbC)/\fra^*\cong\bbT(\bbC)$, where $z=(z_\tau)\in F\otimes\bbC\cong\bbC^I$
and $\fra^*\coloneqq\{\beta\in F\mid \Tr(\beta\fra)\subset\bbZ\}$.
See \cite{BHY19}*{\S 4} for details.  
The functions $\Li^\p_{\bsk,\sigma}(t)$ for $\bsk\in\bbZ^I$ satisfy the following differential equations.

\begin{proposition}\label{prop: diff eq}
	Let $\fra\in{\frI}$ and $\bsalpha\in\sA_\fra^{g}$.
	For $\bsk=(k_\tau)\in\bbZ^I$,
	let $\Li^\p_{\bsk,\sigma}(t)$ be the $p$-adic polylogarithm function 
	for the cone $\sigma=\sigma_\bsalpha$ on the rigid analytic space
	$\wh U^\fra_{\bsalpha K}$.  Then we have
	\[
		\partial_{\tau}\Li^\p_{\bsk,\sigma}(t) = \Li^\p_{\bsk-1_\tau,\sigma}(t),
	\]
	where $1_\tau\in\bbZ^I$ is the element with $1$ in the $\tau$-component and 
	zero in the other components.
\end{proposition}

\begin{proof}
	By the formula $\partial_{\tau}(t^\alpha)=\alpha^\tau t^\alpha$ of Lemma \ref{lem: differential}, 
	we can define the operator $\partial_\tau$ on the ring 
	$K\llbracket t^\alpha\mid\alpha\in\fra_+\rrbracket$, 
	and then on $\wh{B}_\bsalpha\otimes K$ through the inclusion \eqref{eq: wh B alpha K}. 
    The statement regarded as an equality in $K\llbracket t^\alpha\mid\alpha\in\fra_+\rrbracket$ 
    follows from the definition \eqref{eq: polylogarithm} of $\Li^\p_{\bsk,\sigma}(t)$. 
    Since the inclusion $\wh{B}_\bsalpha\otimes K\hookrightarrow \wh{A}_\bsalpha\otimes K$
    is compatible with $\partial_\tau$, we obtain the result. 
\end{proof}

%%%%%%%%%%%%%%%%%%%%%%%%%%%%%%%%%%%%%%%%%%%%%%%%%%
%
\subsection{The $p$-adic Polylogarithm}\label{subsection: p-adic polylog}
%
%%%%%%%%%%%%%%%%%%%%%%%%%%%%%%%%%%%%%%%%%%%%%%%%%%

In this subsection,
we show that the $p$-adic polylogarithm functions for various fractional ideals and cones
combine together to give the $p$-adic polylogarithm.

We consider the rigid analytic spaces $\wh{\bbT}_K\coloneqq\coprod_{\fra\in\frI}\wh{\bbT}^\fra_K$ and 
$\wh{U}_K\coloneqq\coprod_{\fra\in\frI}\wh{U}^\fra_K$ on which $F^\times_+$ naturally acts. 
We first define the $F_+^\times$-equivariant sheaf 
$\sO_{\wh\bbT_K}(\bsk)$ on $\wh\bbT_K$ as follows.

\begin{definition}\label{def: twist}
	For any $\bsk=(k_{\tau})\in\bbZ^I$, we define an $F_+^\times$-equivariant sheaf 
	$\sO_{\wh{\bbT}_K}(\bsk)$ on $\wh\bbT_K$
	as follows.
	As an $\sO_{\wh\bbT_K}$-module we let
	$\sO_{\wh\bbT_K}(\bsk)\coloneqq\sO_{\wh\bbT_K}$, 
	and we define the $F_+^\times$-equivariant structure
	\[
		\iota_x\colon\bra{x}^*\sO_{\wh\bbT_K}(\bsk)\cong\sO_{\wh\bbT_K}(\bsk)
	\]
	to be the multiplication by $x^{-\bsk}\coloneqq
	\prod_{\tau\in I}(x^{\tau})^{-k_{\tau}}$ for any 
	$x\in F_+^\times$.
\end{definition}

The rigid analytic space $\wh{U}_K$ has an affinoid open covering 
$\wh{\frU}_K\coloneqq\{\wh{U}^\fra_{\bsalpha K}\}_{\fra\in\frI,\bsalpha\in\sA_\fra}$. 
Then, similarly to Proposition \ref{prop: standard}, the \v{C}ech complex 
$C^{\bullet}\bigl(\wh\frU_K/F_+^\times,\sO_{\wh\bbT_K}(\bsk)\bigr)$ is defined by
\[
		C^{q}\bigl(\wh\frU_K/F_+^\times,\sO_{\wh\bbT_K}(\bsk)\bigr)
		\coloneqq\biggl(\prod_{\fra\in{\frI}}\prod^\alt_{\bsalpha\in\sA_\fra^{q+1}}
		\Gamma\bigl(\wh U^\fra_\bsalpha,\sO_{\wh\bbT^\fra_K}(\bsk)\bigr)\biggr)^{F_+^\times}
\]
with the standard alternating derivative as in \eqref{eq: standard}. 
Then we have an isomorphism 
\begin{equation}\label{eq: Cech cohomology: rigid}
	H^m(\wh U_K/F_+^\times, \sO_{\wh\bbT_K}(\bsk))
	\cong H^m\bigl(C^{\bullet}\bigl(\wh\frU_K/F_+^\times,\sO_{\wh\bbT_K}(\bsk)\bigr)\bigr), 
\end{equation}
where the left hand side is defiend as the right derived functor of 
$\Gamma(\widehat{U}_K/F^\times_+,-)=\Gamma(\widehat{U}_K,-)^{F^\times_+}$ 
on the category of $F^\times_+$-equivariant $\sO_{\widehat{U}_K}$-modules 
(the fact that this category has enough injectives can be proved in the same way as \cite{Gro57}*{Proposition 5.1.1}). 
The isomorphism \eqref{eq: Cech cohomology: rigid} follows along the same line as Proposition \ref{prop: standard}, 
with the affine vanishing replaced by the affinoid vanishing. 
We also need, instead of \cite{BHY19}*{Lemma 3.5}, the following: 

\begin{lemma}\label{lem: cohomology and product: rigid}
Let $X$ be a rigid analytic space and let $(\sF_\lambda)_{\lambda\in\Lambda}$ be a family of coherent sheaves on $X$. 
Then for any $m\ge 0$, we have 
\[H^m\biggl(X,\prod_{\lambda\in\Lambda}\sF_\lambda\biggr)\cong \prod_{\lambda\in\Lambda}H^m(X,\sF_\lambda), \]
where $\prod_{\lambda\in\Lambda}\sF_\lambda$ is the product in the category of $\sO_X$-modules. 
\end{lemma}
\begin{proof}
We take an injective resolution $0\to \sF_\lambda\to I_\lambda^\bullet$ for each $\lambda$, and 
will show that $0\to\prod_\lambda\sF_\lambda\to \prod_\lambda I_\lambda^\bullet$ gives an injective resolution. 
Since each component $\prod_\lambda I_\lambda^q$ is injective, it suffices to show the exactness. 
Recall that a complex of $\sO_X$-modules is exact if and only if the complexes of stalks at 
all prime filters are exact \cite{vdPS95}*{p.~94}. 

For any prime filter $x$ on $X$, the collection of admissible affinoid open subsets belonging to $x$ is 
cofinal in the directed set $x$. On the other hand, for any admissible affinoid open subset $V$ of $X$, 
the affinoid vanishing implies that 
$0\to \sF_\lambda(V)\to I_\lambda^\bullet(V)$ is exact for each $\lambda$, and so is the product 
$0\to \prod_\lambda\sF_\lambda(V)\to \prod_\lambda I_\lambda^\bullet(V)$. 
Thus, by passing to the direct limit, we see that 
$0\to (\prod_\lambda\sF_\lambda)_x\to (\prod_\lambda I_\lambda^\bullet)_x$ is also exact. 
This completes the proof. 
\end{proof}

\begin{theorem}\label{thm: polylog}
		Let $\bsk\in\bbZ^I$.
		For any $\fra\in\frI$ and $\bsalpha\in\sA_\fra^g$,
		we let
		\[
			\Li^\p_{\bsk,\bsalpha}(t)\coloneqq
			\sgn(\bsalpha)\Li^\p_{\bsk,\sigma_\bsalpha}(t)
			\in\Gamma(\wh U^\fra_\bsalpha,\sO_{\wh\bbT^\fra_K}(\bsk)),
		\]		
		where $\Li^\p_{\bsk,\sigma_\bsalpha}(t)$ is the 
		$p$-adic polylogarithm function for the cone $\sigma_\bsalpha$, 
		regarded as an element of $\Gamma(\wh U^\fra_\bsalpha,\sO_{\wh\bbT^\fra_K}(\bsk))$ 
		via Proposition \ref{prop: continuation}.
		Then $\big(\Li^\p_{\bsk,\bsalpha}(t)\big)$ gives a cocycle in
		$C^{q}\bigl(\wh\frU_K/F_+^\times,\sO_{\wh\bbT_K}(\bsk)\bigr)$,
		hence defines a class 
		\[
			\Li^\p_\bsk(t)\in H^{g-1}\bigl(\wh U_K/F_+^\times,\sO_{\wh\bbT_K}(\bsk)\bigr),
		\]
		which we call the $p$-adic polylogarithm.
\end{theorem}

\begin{proof}
	By construction, 
	$\bigl(\sgn(\bsalpha)\Li^\p_{\bsk,\sigma_\bsalpha}(t)\bigr)$ defines an element in 
	$\prod_{\fra\in\frI}
	\prod^\alt_{\bsalpha\in\sA_\fra^{g}}
	\Gamma\bigl(\wh U^\fra_{\bsalpha K},\sO_{\wh\bbT^\fra_K}(\bsk)\bigr)$.
	For any $x\in F_+^\times$, the function $\Li^\p_{\bsk,\sigma}(t)$ satisfies
	\[
		\bra{x}^*\Li^\p_{\bsk,\sigma}(t)
		=\sum_{\substack{\alpha\in\breve\sigma\cap\fra\\
		\alpha\in(\fra\otimes\bbZ_p)^\times}}\frac{1}{\alpha^\bsk}t^{x\alpha}
		=x^{\bsk}\sum_{\substack{\alpha\in x(\breve\sigma\cap\fra)\\
		\alpha\in(x\fra\otimes\bbZ_p)^\times}}\frac{1}{\alpha^\bsk}t^{\alpha}
		=x^{\bsk}\Li^\p_{\bsk,x\sigma}(t) 
		\in \Gamma(\wh U^{x\fra}_{x\bsalpha},\sO_{\wh\bbT^{x\fra}_K}(\bsk)),
	\]
	hence we see through the isomorphism 
	$\iota_x\colon\langle x\rangle^*\sO_{\wh\bbT_K}(\bsk)\cong\sO_{\wh\bbT_K}(\bsk)$ 
	as in Definition \ref{def: twist} that the element 
	\[
		\Li_\bsk^\p(t)=\big(\sgn(\bsalpha)\Li^\p_{\bsk,\sigma_\bsalpha}(t)\big)
		\in \prod_{\fra\in\frI}
	\prod^\alt_{\bsalpha\in\sA_\fra^{g}}\Gamma\bigl(\wh U^\fra_\bsalpha,\sO_{\wh\bbT^\fra_K}(\bsk)\bigr)
	\] 
	is invariant with respect to the action of $F_+^\times$, that is
	$\Li^\p_{\bsk}(t)\in C^{g-1}\bigl(\wh\frU_K/F_+^\times,\sO_{\wh\bbT_K}(\bsk)\bigr)$.
	The fact that  $\Li^\p_{\bsk}(t)$ is a cocycle 
	follows from the linear relation of characteristic functions
	\[
		\sum_{j=0}^g (-1)^j 
		\sgn(\alpha_0,\ldots,\breve\alpha_j,\ldots,\alpha_g) 
		1_{\breve\sigma_{\alpha_0,\ldots,\breve\alpha_j,\ldots,\alpha_g}}=0
	\]
	for any $(\alpha_0,\ldots,\alpha_g)\in\sA_\fra^{g+1}$ given in \cite{Yam10}*{Proposition 6.2}.
	 This proves that 
	$
		\Li^\p_{\bsk}(t)
	$
	defines a class in $H^{g-1}(\wh U_K/F_+^\times,\sO_{\wh\bbT}(\bsk))$ as desired.
\end{proof}

The differential operator $\partial_\tau$ on $\wh\bbT^\fra_K$ for $\fra\in\frI$ gives, 
for each $\bsk\in\bbZ^I$, a morphism of abelian sheaves 
\begin{equation*}%\label{eq: differential}
	\partial_\tau\colon\sO_{\wh\bbT_K}(\bsk)\rightarrow\sO_{\wh\bbT_K}(\bsk-1_\tau)
\end{equation*}
compatible with the action of $F_+^\times$ 
(recall that $1_\tau\in\bbZ^I$ is the element with $1$ in the $\tau$-th component 
and $0$ in the other components).
This induces a homomorphism
\[
	\partial_\tau\colon H^m\bigl(\wh U_K/F_+^\times,\sO_{\wh\bbT_K}(\bsk)\bigr)\rightarrow 
	H^m\bigl(\wh U_K/F_+^\times,\sO_{\wh\bbT_K}(\bsk-1_\tau)\bigr)
\]
on equivariant cohomology.  
The following formula is an immediate consequence of Proposition \ref{prop: diff eq}.

\begin{proposition}%\label{prop: differential equation}
	For any $\bsk\in\bbZ^I$, the $p$-adic polylogarithm class satisfies the differential equation
	\[
		\partial_\tau\Li^\p_{\bsk}(t) = \Li^\p_{\bsk-1_\tau}(t).
	\]
\end{proposition}

%%%%%%%%%%%%%%%%%%%%%%%%%%%%%%%%%%%%%%%%%%%%%%%%%%
%
%
%
\section{$p$-adic Interpolation of Lerch Zeta Values}\label{section: interpolation}
%
%
%
%%%%%%%%%%%%%%%%%%%%%%%%%%%%%%%%%%%%%%%%%%%%%%%%%%

The purpose of this section is to construct a $p$-adic measure interpolating the special 
values of Lerch zeta functions.  We will then use this measure to construct the $p$-adic
$L$-functions of Hecke characters of the totally real field $F$.

%%%%%%%%%%%%%%%%%%%%%%%%%%%%%%%%%%%%%%%%%%%%%%%%%%
%
\subsection{The $p$-adic interpolations of Lerch zeta and Hecke $L$-values}
%\label{subsection: p-adic interpolation}
%
%%%%%%%%%%%%%%%%%%%%%%%%%%%%%%%%%%%%%%%%%%%%%%%%%%

In this subsection, we will state the main results of this section, 
concerning certain $p$-adic measures which interpolate 
the values of Lerch zeta functions and Hecke $L$-functions at nonpositive integers.
The actual constructions of such measures will be given in \S\ref{subsection: Lerch zeta} 
and \S\ref{subsection: padicHecke}. 

First, we consider the quotient topological space $\ol{\Delta}\backslash(\fra\otimes\bbZ_p)$ 
of $\fra\otimes\bbZ_p$ by the action of the closure $\ol{\Delta}$ of $\Delta$ 
in $\cO_F\otimes\bbZ_p$. This space may be written as 
\[
    \ol{\Delta}\backslash(\fra\otimes\bbZ_p)
    =\varprojlim_m \Delta\backslash(\fra/p^m\fra). 
\]
We define the \emph{norm} function 
$\wh{N}_\fra\colon\ol{\Delta}\backslash(\fra\otimes\bbZ_p)\to\bbZ_p$
as the limit of the functions 
\[
    \Delta\backslash(\fra/p^m\fra)\longrightarrow \bbZ_p/p^m\bbZ_p;\ 
    \alpha\longmapsto N(\fra^{-1}\alpha)\bmod p^m, 
\]
where the representative $\alpha$ is chosen from $\fra_+$. 
In other words, $\wh{N}_\fra$ is a continuous function determined by 
$\wh{N}_\fra(\alpha\otimes 1)=N(\fra^{-1}\alpha)$ for $\alpha\in\fra_+$. 
Then we have an identity 
\begin{equation}\label{eq: N_a}
    \wh{N}_\fra(x)=N\fra^{-1} N(x)  
\end{equation}
for any $x\in\fra\otimes\bbZ_p$, where we let
\begin{equation}\label{eq: N(x)}
	N(x)\coloneqq\prod_{\tau\in I}x^\tau.
\end{equation}
Here, $x^\tau$ for $\tau\in I$ denotes the image of $x$
with respect to the map $\tau\colon\fra\otimes\bbZ_p\rightarrow K$ 
induced by $\alpha\otimes c\mapsto\alpha^\tau c$.
Note that such map is defined since we have taken $K$ sufficiently large
so that $\tau(F)\subset K$ for any $\tau\in I$.

We say that a function $\phi$ on a fractional ideal $\fra$ is \textit{periodic} 
if there exists a nonzero integral ideal $\frg$ of $F$ such that 
$\phi$ factors through $\fra/\frg\fra$.
If a periodic function $\phi\colon\fra\rightarrow\bbC$ is invariant under the action of $\Delta$, 
we define 
\[
	\cL(\phi,s)\coloneqq
	\sum_{\alpha\in\Delta\backslash\fra_+} \phi(\alpha) N(\fra^{-1}\alpha)^{-s}. 
\]
For $\phi=\xiDelta$, this recovers the definition of the Lerch zeta function $\cL(\xiDelta,s)$.

\begin{theorem}[$p$-adic interpolation of Lerch zeta values]\label{thm: interpolation}
    Let $\fra\in\frI$ and $\xi$ be a torsion point of $\bbT^\fra$ 
    whose order is not a $p$-power. Then there exists a unique measure 
    $\mu_{\xiDelta}$ on $\ol{\Delta}\backslash(\fra\otimes\bbZ_p)$ 
    satisfying the interpolation property 
    \[
		\int_{\ol\Delta\backslash(\fra\otimes\bbZ_p)}
		\xipDelta(x)\wh{N}_\fra(x)^k d\mu_{\xiDelta}(x)
		=\cL(\xiDelta\,\xipDelta,-k)
	\]
	for any $\xi_p\in\bbT^\fra[p^\infty]\coloneqq\bigcup_{m\geq 0}\bbT^\fra[p^m]$ 
	and any $k\in\bbN$. 
\end{theorem}

Next, let us consider the projective limit 
\[
    \Cl^+_F(p^\infty)\coloneqq \varprojlim_m \Cl^+_F(p^m) 
\]
of the narrow ray class groups of $p$-power moduli, 
with respect to the natural homomorphisms $\Cl^+_F(p^{m+1})\to\Cl^+_F(p^m)$. 
By taking the limit, Lemma \ref{lem: ray class group} gives a map 
\[
    (\fra\otimes\bbZ_p)^\times\longrightarrow \Cl^+_F(p^\infty)
\]
for each $\fra\in\frI$, which we often denote by $\alpha\longmapsto\fra^{-1}\alpha$ 
by abuse of notation (precisely speaking, it makes sense only for $\alpha\in\fra_+$). 
Moreover, if we choose a set of representatives $\frC\subset\frI$ for $\Cl^+_F(1)$, 
we have a bijection 
\begin{equation}\label{eq: Cl(p^infty)}
    \coprod_{\fra\in\frC}\ol{\Delta}\backslash(\fra\otimes\bbZ_p)^\times
    \longrightarrow \Cl^+_F(p^\infty), 
\end{equation}
where $\ol{\Delta}$ denotes the closure of $\Delta$ in $\cO_F\otimes\bbZ_p$. 

The \emph{norm} on $\Cl^+_F(p^\infty)$ is the homomorphism 
\[
    \wh{N}\colon\Cl^+_F(p^\infty)\longrightarrow\bbZ_p^\times
\]
defined as the limit of the maps 
\[
    \Cl^+_F(p^m)\longrightarrow(\bbZ/p^m\bbZ)^\times;\ 
    \fra\longmapsto N(\fra)\bmod p^m. 
\]
This $\wh{N}$ is compatible with the norm function $\wh{N}_\fra$ for $\fra\in\frI$ 
through the inclusion $\ol{\Delta}\backslash(\fra\otimes\bbZ_p)^\times\to\Cl^+_F(p^\infty)$, 
i.e., we have $\wh{N}(\fra^{-1}\alpha)=\wh{N}_\fra(\alpha)$ for 
$\alpha\in\ol{\Delta}\backslash(\fra\otimes\bbZ_p)^\times$. 

\begin{theorem}[$p$-adic interpolation of Hecke $L$-values]\label{thm: padicL}
	Let $\chi\colon\Cl^+_F(\frg)\rightarrow\bbC^\times$ be a finite primitive 
	Hecke character of conductor $\frg$, where $\frg$ does not divide any power of $(p)$. 
	Then there exists a unique measure $\mu_\chi$ on $\Cl^+_F(p^\infty)$ 
	satisfying the interpolation property
	\[
		\int_{\Cl^+_F(p^\infty)} \chi_p(\frx)\wh{N}(\frx)^k d\mu_\chi(\frx) 
		= \Biggl( \prod_{\frp\divides(p)}\bigl(1-\chi\!\chi_p(\frp)N\frp^{-k}\bigr)\Biggr)
		L(\chi\!\chi_p,-k)
	\]
	for any finite character $\chi_p\colon\Cl^+_F(p^\infty)\rightarrow\ol\bbQ^\times$ 
	and any $k\in\bbN$. 
\end{theorem}

\begin{remark}\label{rem: chi chi_p}
In the setting of the above theorem, the product $\chi\!\chi_p$ 
of Hecke characters $\chi$ and $\chi_p$ defines a character of $\Cl^+_F(\frg')$, 
where $\frg'$ is some integral ideal which differs from $\frg$ 
only by some powers of the prime divisors of $(p)$. 
Note that the right hand side of the formula in Theorem \ref{thm: padicL} is 
independent of the choice of $\frg'$, 
since the Euler factors at the prime divisors of $(p)$ are deleted. 
In the following, we take $\frg'$ as large (with respect to inclusion of sets) as possible, 
and thus regard $\chi\!\chi_p$ as a primitive Hecke character. 
\end{remark}

Note that, via the bijection \eqref{eq: Cl(p^infty)}, 
giving a measure on $\Cl^+_F(p^\infty)$ is equivalent to giving 
a collection of measures on $\ol{\Delta}\backslash(\fra\otimes\bbZ_p)^\times$ 
with $\fra$ running over a set of representatives $\frC$ for $\Cl^+_F(1)$. 
Roughly speaking, the measure $\mu_\chi$ of Theorem \ref{thm: padicL} 
will be constructed from the measures $\mu_{\xiDelta}$ of Theorem \ref{thm: interpolation} 
through this correspondence, since the Hecke $L$-function is 
a linear combination of the Lerch zeta functions (Proposition \ref{prop: Hecke primitive}). 
The detailed construction will be given in \S\ref{subsection: padicHecke}. 

On the other hand, the construction of $\mu_{\xiDelta}$ is based on 
the $p$-adic interpolation of the Shintani zeta functions, 
on which we discuss in \S\ref{subsection: Shintani zeta}. 
A Shintani zeta function is associated with a cone, and we can express 
the Lerch zeta function in terms of Shintani zeta functions for 
a certain finite collection of cones, called a Shintani decomposition. 
See \S\ref{subsection: Lerch zeta} for the detail.

%%%%%%%%%%%%%%%%%%%%%%%%%%%%%%%%%%%%%%%%%%%%%%%%%%
%
\subsection{Shintani zeta functions and $p$-adic interpolation}\label{subsection: Shintani zeta}
%
%%%%%%%%%%%%%%%%%%%%%%%%%%%%%%%%%%%%%%%%%%%%%%%%%%

In this subsection, we review the definition of the Shintani zeta function, 
and will use the generating function of its values at nonpositive integers 
to construct a $p$-adic measure interpolating these values.

\begin{definition}%\label{def: Shintani zeta}
	Let $\fra$ be a nonzero fractional ideal of $F$. 
    For a periodic function $\phi\colon\fra\to\bbC$ and a cone $\sigma$, 
    we define the \textit{Shintani zeta function} $\zeta_\sigma(\phi,\bss)$ by 
	\begin{equation}\label{eq: Shintani zeta}
		\zeta_\sigma(\phi,\bss)\coloneqq
		\sum_{\alpha\in\breve\sigma\cap\fra} \phi(\alpha) \alpha^{-\bss},
	\end{equation}
	where $\bss=(s_{\tau})\in\bbC^I$ and 
	$\alpha^{-\bss}\coloneqq\prod_{\tau\in I}(\alpha^{\tau})^{-s_{\tau}}$. 
	Moreover, for a single variable $s\in\bbC$, we let 
\[
	\zeta_\sigma(\phi,s)\coloneqq\zeta_\sigma(\phi,(s,\ldots,s)) 
	= \sum_{\alpha\in\breve\sigma\cap\fra} \phi(\alpha)N(\alpha)^{-s}.
\]
\end{definition}

The series \eqref{eq: Shintani zeta} converges if $\Re(s_{\tau})>1$ for any $\tau\in I$.
By \cite{Shi76}*{Proposition 1}, this function has a meromorphic continuation 
to the whole space $\bbC^I$.

%Let $\fra$ be a nonzero fractional ideal of $F$.
%We will first construct a $p$-adic measure on $\fra\otimes\bbZ_p$
%interpolating values at nonpositive integral points of Shintani zeta functions $\zeta_{\sigma}(\xi,\bss)$. 

For any $\tau\in I$, we let $\partial_\tau$ be the differential operator 
defined in \eqref{eq: differential operator}, and 
write $\partial^\bsk\coloneqq\prod_{\tau\in I}\partial_\tau^{k_\tau}$ 
for any $\bsk=(k_\tau)\in\bbN^I$.
%For any $\alpha\in\fra$, we let $U^\fra_\alpha=\bbT^\fra\setminus\{t^\alpha=1\}$,
%which is an affine open set in $\bbT^\fra$. 
The following theorem, based on the work of Shintani, is standard (see for example
\cite{CN79}*{Th\'eor\`eme 5}, \cite{Col88}*{Lemme 3.2}, \cite{BHY19}*{Proposition 4.3}).

\begin{theorem}\label{thm: generate}
	Let $\fra\in\frI$, and let $\sigma$ be a $g$-dimensional cone 
	generated by $\bsalpha\in\sA_\fra^g$.
	Then, for any torsion point $\xi$ in $U^\fra_\bsalpha(\ol{\bbQ})$, we have
	\[
		\partial^\bsk\cG^\fra_{\sigma}(t)\big|_{t=\xi} = \zeta_{\sigma}(\xi,-\bsk), 
	\]
	where $\cG^\fra_\sigma(t)\in\Gamma(U^\fra_\bsalpha,\sO_\bbT)$ is the function 
	defined in Definition \ref{def: G}.
\end{theorem}

Let $\xi\in\bbT^\fra(\ol\bbQ)$ be a torsion point lying in $U^\fra_\bsalpha$. 
Then we may expand the rational function $\cG^\fra_\sigma$ around $\xi$ to obtain a power series
\[
	\wt{\cG}^\fra_{\sigma,\xi}(T)\in\bbQ(\xi)\llbracket T^\alpha\mid\alpha\in\fra\rrbracket,
\]
where $\bbQ(\xi)\llbracket T^\alpha\mid\alpha\in\fra\rrbracket$ denotes 
the completion of $\bbQ(\xi)[t^\alpha\mid\alpha\in\fra]$ 
with respect to the ideal generated by $T^\alpha\coloneqq\xi(\alpha)^{-1}t^\alpha-1$ 
for all $\alpha\in\fra$. Note that, if we choose a $\bbZ$-basis $\gamma_1,\ldots,\gamma_g$ of $\fra$, 
we may identify this ring with the usual ring 
$\bbQ(\xi)\llbracket T^{\gamma_1},\ldots,T^{\gamma_g}\rrbracket$ of formal power series. 

Next, we further assume that the torsion point $\xi$ lies in $\wh{U}^\fra_{\bsalpha K}$, 
i.e., $\xi(\alpha_j)\not\equiv 1$ modulo the maximal ideal of $O_{K(\xi)}$ for $j=1,\ldots,g$. 
Then we see that  
\[
	\wt\cG^\fra_{\sigma,\xi}(T)\in\cO_{K(\xi)}\llbracket T^\alpha\mid \alpha\in\fra\rrbracket
\]
by observing that 
\[
    \frac{1}{1-t^{\alpha_j}}
    =\frac{1}{1-\xi(\alpha_j)(1+T^{\alpha_j})}
    =\sum_{n=0}^\infty\frac{(\xi(\alpha_j)T^{\alpha_j})^n}{(1-\xi(\alpha_j))^{n+1}}
    \in\cO_{K(\xi)}\llbracket T^\alpha\mid \alpha\in\fra\rrbracket 
\]
for $j=1,\ldots,g$. 
By Mahler's theorem, this power series defines a $p$-adic measure $\mu_{\sigma,\xi}$ 
on $\fra\otimes\bbZ_p$ satisfying the following conditions.

\begin{proposition}\label{prop: interpolation}
	Let $\xi$ be a torsion point lying in $\wh{U}^\fra_{\bsalpha K}$ 
	and put $\sigma=\sigma_\bsalpha$. 
	Then there exists a unique measure $\mu_{\sigma,\xi}$ on $\fra\otimes\bbZ_p$ satisfying
	the interpolation property
	\begin{equation}\label{eq: interpolation}
		\int_{\fra\otimes\bbZ_p} \xi_p(x) x^\bsk d\mu_{\sigma,\xi}(x) = \zeta_{\sigma}(\xi\xi_p,-\bsk)
	\end{equation}
	for any $\xi_p\in\bbT^\fra[p^\infty]\coloneqq\bigcup_{m\in\bbN}\bbT^\fra[p^m]$ 
	and $\bsk\in\bbN^I$. 
\end{proposition}

\begin{proof}
    The uniqueness follows from the fact that the elements of $\bbT^\fra[p^\infty]$ 
    spans a dense subspace in the space of continuous functions on $\fra\otimes\bbZ_p$. 
    To construct the measure, we fix a $\bbZ$-basis $\bsgamma=(\gamma_1,\ldots,\gamma_g)$ of $\fra$, 
	which gives an isomorphism $\bbZ_p^g\cong\fra\otimes\bbZ_p$ of $\bbZ_p$-modules 
	mapping $\bsx=(x_1,\ldots,x_g)$ to $\bsx\cdot\bsgamma\coloneqq x_1\gamma_1+\cdots+x_g\gamma_g$.
	If we let $T_j\coloneqq \xi(\gamma_j)^{-1}t^{\gamma_j}-1$ for $j=1,\ldots,g$, 	
	then we have 
	\[
		\wt\cG^\fra_{\sigma,\xi}(T)\in\cO_{K(\xi)}\llbracket T_1,\ldots,T_g\rrbracket,
	\]
	and Mahler's Theorem (see for example \cite{Hid93}*{\S3.7 Theorem 1}) gives a $p$-adic measure 
	$\mu_{\sigma,\xi}^{\bsgamma}$ on $\bbZ_p^g$ satisfying
	\begin{align*}
		\int_{\bbZ_p^g}  (1+T_1)^{x_1}\cdots (1+T_g)^{x_g}
		d\mu_{\sigma,\xi}^{\bsgamma}(\bsx)&=
		\wt\cG^\fra_{\sigma,\xi}(T)
		\in\cO_{K(\xi)}\llbracket T_1,\ldots,T_g\rrbracket.
	\end{align*}
	Then the measure $\mu_{\sigma,\xi}^{\bsgamma}$ satisfies	
	\begin{align*}
		\int_{\bbZ_p^g} \xi_p(\bsx\cdot\bsgamma) x_1^{n_1}\cdots x_g^{n_g}
		d\mu_{\sigma,\xi}^{\bsgamma}(\bsx)&=
		 \partial_{\log  T_1}^{n_1}\cdots
		\partial_{\log T_g}^{n_g}\wt\cG^\fra_{\sigma,\xi}(T)
		\big|_{(T_j) =(\xi_p(\gamma_j)-1)}\\
		&=
		 \partial_{\log \gamma_1}^{n_1}\cdots
		\partial_{\log \gamma_g}^{n_g}\cG^\fra_{\sigma}(t)
		\big|_{t =\xi\xi_p},
	\end{align*}
	for any $(n_j)\in\bbN^g$, where $\partial_{\log T_j}\coloneqq(1+T_j)\frac{d}{dT_j}$.
	By \eqref{eq: differential operator} and Theorem \ref{thm: generate}, we see that 
	the $p$-adic measure $\mu_{\sigma,\xi}$ on $\fra\otimes\bbZ_p$ induced by 
	$d\mu_{\sigma,\xi}^{\bsgamma}$ through the isomorphism $\bbZ_p^g\cong\fra\otimes\bbZ_p$ 
	satisfies the interpolation property \eqref{eq: interpolation}. 
\end{proof}

\begin{remark}\label{rem: mu_sigma,xi}
    Let $x\in F^\times_+$, and consider the measures 
    $\mu_{\sigma,\xi^x}$ on $\fra\otimes\bbZ_p$ and $\mu_{x\sigma,\xi}$ on $x\fra\otimes\bbZ_p$, 
    where $\xi$ is a torsion point lying in $\wh{U}^{x\fra}_{x\bsalpha K}$ 
    with $\bsalpha\in\sA_\fra^g$ and $\sigma=\sigma_\bsalpha$. 
    Then the push-forward of $\mu_{\sigma,\xi^x}$ with respect to the multiplication by $x$ 
    is equal to $\mu_{x\sigma,\xi}$. This follows either from the equivariance 
    $\bra{x}^*\cG^\fra_\sigma=\cG^{x\fra}_{x\sigma}$ of generating functions, 
    or from the interpolation property \eqref{eq: interpolation}. 
\end{remark}

%%%%%%%%%%%%%%%%%%%%%%%%%%%%%%%%%%%%%%%%%%%%%%%%%%
%
\subsection{$p$-adic Interpolation of Lerch Zeta Values}\label{subsection: Lerch zeta}
%
%%%%%%%%%%%%%%%%%%%%%%%%%%%%%%%%%%%%%%%%%%%%%%%%%%

The Lerch zeta functions may be expressed in terms of Shintani zeta functions
via a choice of a Shintani decomposition.
We will construct a $p$-adic measure $\mu_{\xiDelta}$
interpolating special values of Lerch zeta values,
using the measures $\mu_{\sigma,\xi}$ constructed in Proposition \ref{prop: interpolation}.

By \cite{Yam10}*{Proposition 5.6}, there exists a set $\Phi$ of $g$-dimensional simplicial cones 
stable under the action of $\Delta$, such that $\Delta\backslash\Phi$ is finite and 
\begin{equation*}%\label{eq: upper closure}
	\bbR_+^I=\coprod_{\sigma\in\Phi}\breve\sigma. 
\end{equation*}
In this article, we call such $\Phi$ a \textit{Shintani decomposition} of $\bbR^I_+$. 
From the definition of the Lerch zeta function in \eqref{eq: Lerch}, we have the following.

\begin{proposition}%\label{prop: shintani}
	Let $\xi$ be a torsion point of $\bbT^\fra(\ol\bbQ)$.
	If $\Phi$ is a Shintani decomposition, then we have 
	\[
		\cL(\xiDelta, s) = N\fra^s\sum_{\sigma\in\Delta\backslash\Phi}\zeta_\sigma(\xiDelta,s).
	\]
\end{proposition}

The following result is obtained by the same argument as in the proof of \cite{BHY19}*{Lemma 5.3}, 
by replacing the condition $\xi(\alpha)\ne 1$ by $\xi(\alpha)\not\equiv 1$ 
modulo the maximal ideal of $O_{K(\xi)}$. 

\begin{lemma}\label{lem: SD}
	Let $\xi$ be a torsion point of $\bbT^\fra$ whose order is not a $p$-power. 
	Then there exists a Shintani decomposition $\Phi_\xi$ of $\bbR^I_+$ such that 
	$\xi$ lies in $\wh{U}^\fra_{\bsalpha K}$ whenever $\bsalpha\in\sA_\fra^g$ 
	is a generator of some $\sigma\in\Phi_\xi$. 
\end{lemma}

Now let us prove Theorem \ref{thm: interpolation}. 

\begin{proof}[Proof of Theorem \ref{thm: interpolation}]
    Let $\xi$ be a torsion point of $\bbT^\fra\setminus\bbT^\fra[p^\infty]$.
    In what follows, we fix a Shintani decomposition $\Phi_\xi$ as in Lemma \ref{lem: SD}, 
    and denote by $\ol{\mu}_{\sigma,\xi}$ for $\sigma\in\Phi_\xi$ 
    the push-forward of the measure $\mu_{\sigma,\xi}$ 
    with respect to the projection $\fra\otimes\bbZ_p\to\ol{\Delta}\backslash(\fra\otimes\bbZ_p)$. 
    Then Remark \ref{rem: mu_sigma,xi} shows that $\ol{\mu}_{\sigma,\xi}$ 
    depends only on the orbit of $\sigma$ in $\Delta_\xi\backslash\Phi_\xi$. 
    Therefore, we may define a measure $\mu_{\xiDelta}$ on $\ol{\Delta}\backslash(\fra\otimes\bbZ_p)$ 
    by 
    \[
        \mu_{\xiDelta}\coloneqq\sum_{\sigma\in \Delta_\xi\backslash\Phi_\xi} \ol{\mu}_{\sigma,\xi}. 
    \]
    If we choose a set of representatives $C_\xi$ of $\Delta_\xi\backslash\Phi_\xi$, we see that 
	\begin{align*}
		\int_{\ol\Delta\backslash(\fra\otimes\bbZ_p)}\xipDelta(x)\wh{N}_\fra(x)^k d\mu_{\xiDelta}(x)
		&=\sum_{\sigma\in C_\xi}\sum_{\e_p\in\Delta_{\xi_p}\backslash\Delta}
		\int_{\fra\otimes\bbZ_p}\xi_p^{\e_p}(x)\wh{N}_\fra(x)^k d\mu_{\sigma,\xi}(x)\\
		&=N\fra^{-k}\sum_{\sigma\in C_\xi}\sum_{\e_p\in\Delta_{\xi_p}\backslash\Delta}
		\zeta_\sigma(\xi\xi_p^{\e_p},-k)
	\end{align*}
	by the identity \eqref{eq: N_a} and Proposition \ref{prop: interpolation}. 
	Since we have 
	\begin{align*}
		N\fra^s\sum_{\sigma\in C_\xi}&\sum_{\e_p\in\Delta_{\xi_p}\backslash\Delta}
		\zeta_\sigma(\xi\xi_p^{\e_p},s)
		=\sum_{\sigma\in C_\xi}
		\sum_{\alpha\in\breve\sigma\cap\fra}
		\sum_{\e_p\in\Delta_{\xi_p}\backslash\Delta}
		\xi(\alpha)\xi_p^{\e_p}(\alpha)N(\fra^{-1}\alpha)^{-s}\\
		&= \sum_{\alpha\in\Delta\backslash\fra_+}
		\sum_{\e_p\in\Delta_{\xi_p}\backslash\Delta}
		\xiDelta(\alpha)\xi_p^{\e_p}(\alpha)N(\fra^{-1}\alpha)^{-s}
		= \sum_{\alpha\in\Delta\backslash\fra_+}
		\xiDelta(\alpha)\xipDelta(\alpha)N(\fra^{-1}\alpha)^{-s}\\ 
		&=\cL(\xiDelta\,\xipDelta,s), 
	\end{align*}
	our assertion follows by analytic continuation.
\end{proof}

We next investigate the restriction of the measure $\mu_{\xiDelta}$ to
$\ol\Delta\backslash(\fra\otimes\bbZ_p)^\times$. 
We first define a $p$-modified variant of the Lerch zeta function.

\begin{definition}
	Let $\fra$ be a fractional ideal of $F$, and let $\xi$ be a torsion point of $\bbT^\fra(\ol\bbQ)$.  
	For any rational prime $p$,
	we define the $p$-modified \textit{Lerch zeta function} for $\xi$ by
	\[
		\cL^\p(\xiDelta,s)\coloneqq
		\sum_{\substack{\alpha\in\Delta\backslash\fra_+\\\alpha\in(\fra\otimes\bbZ_p)^\times}}
		\xiDelta(\alpha) N(\fra^{-1}\alpha)^{-s}.
	\]
\end{definition}
We note that for any prime ideal $\frp$ of $F$ such that $\frp\divides(p)$, we have
\[
	\boldsymbol{1}_{\frp}(x)\coloneqq
	\frac{1}{N\frp} \sum_{\xi_\frp\in\bbT^\fra[\frp]} \xi_\frp(x) = 
	\begin{cases}
		1 & x\in \frp\fra \otimes\bbZ_\frp,\\
		0 & x\not\in \frp\fra \otimes\bbZ_\frp.
	\end{cases}
\]
This shows that
\[	
	\boldsymbol{1}_{(\fra\otimes\bbZ_\frp)^\times}(x)\coloneqq
	\prod_{\frp\divides(p)}
	\biggl(1 - \frac{1}{N\frp} \sum_{\xi_\frp\in\bbT^\fra[\frp]/\Delta} \xi_\frp\Delta(x) \biggr)
	=\begin{cases}
		0  &  x\not\in (\fra\otimes\bbZ_\frp)^\times,\\
		1 & x\in (\fra\otimes\bbZ_\frp)^\times,
	\end{cases}
\]
which induces a function on $\ol\Delta\backslash(\fra\otimes\bbZ_p)$.
In what follows, let $P\coloneqq\{\frp\in\Spec\cO_F\mid\frp\divides(p)\}$ be the
set of prime ideals of $\cO_F$ dividing $(p)$ and let
$\frp_J\coloneqq\prod_{\frp\in J}\frp$ for any $J\subset P$.
Theorem \ref{thm: interpolation} gives the following corollary.

\begin{corollary}
	Let $\xi$ be a torsion point of $\bbT^\fra\setminus\bbT^\fra[p^\infty]$.
	The restriction of the measure $\mu_{\xiDelta}$ to $\ol\Delta\backslash(\fra\otimes\bbZ_p)^\times$
	satisfies the interpolation property
	\[
		\int_{\ol\Delta\backslash(\fra\otimes\bbZ_p)^\times}
		\xipDelta(x)\wh{N}_\fra(x)^k d\mu_{\xiDelta}(x)
		=\cL^\p(\xiDelta\,\xipDelta,-k)
	\]
	for any $\xi_p\in\bbT^\fra[p^\infty]$ and $k\in\bbN$.
\end{corollary}

\begin{proof}
	By choosing a set of representatives $C_\xi$ of $\Delta_\xi\backslash\Phi_\xi$, we have
	\begin{align*}
    	\int_{\ol\Delta\backslash(\fra\otimes\bbZ_p)^\times}
    	\xipDelta(x)& \wh{N}_\fra(x)^k d\mu_{\xiDelta}(x)
    	= \int_{\ol\Delta\backslash(\fra\otimes\bbZ_p)}	
    	\boldsymbol{1}_{(\fra\otimes\bbZ_\frp)^\times}(x)\xipDelta(x)
    	\wh{N}_\fra(x)^k
    	d\mu_{\xiDelta}(x)\\
    	&=\sum_{\sigma\in C_\xi}\sum_{J\subset P}\frac{(-1)^{|J|}}{N\frp_J}
    	\sum_{\xi_{\frp_J}\in\bbT^\fra[\frp_J]}\sum_{\e\in\Delta_{\xi_p}\backslash\Delta}
    	\int_{\fra\otimes\bbZ_p}\xi_{\frp_J}(x)\xi^\e_p(x)\wh{N}_\fra(x)^k d\mu_{\sigma,\xi}(x)\\
    	&=N\fra^{-k}\sum_{\sigma\in C_\xi}\sum_{\e\in\Delta_{\xi_p}\backslash\Delta}
    	\sum_{J\subset P}\frac{(-1)^{|J|}}{N\frp_J}
    	\sum_{\xi_{\frp_J}\in\bbT^\fra[\frp_J]}
    	\zeta_\sigma(\xi\xi_{\frp_J}\xi_p^\e, -k).	
	\end{align*}
	Note that for any torsion point $\xi$ of $\bbT^\fra$, we have
	\begin{align*}
		\sum_{J\subset P}\frac{(-1)^{|J|}}{N\frp_J}
		\sum_{\xi_{\frp_J}\in\bbT^\fra[\frp_J]}
		\zeta_\sigma(\xi\xi_{\frp_J}\xi^\e_p, s)
		&=	
		\sum_{\alpha\in\breve\sigma\cap\fra_+}
		\sum_{J\subset P}\frac{(-1)^{|J|}}{N\frp_J}
		\sum_{\xi_{\frp_J}\in\bbT^\fra[\frp_J]}
		\xi_{\frp_J}(\alpha)\frac{\xi(\alpha)\xi_p^\e(\alpha)}{N(\fra^{-1}\alpha)^s}\\
		&=	
		\sum_{\substack{\alpha\in\breve\sigma\cap\fra_+\\\alpha\not\in\frp\fra\,\forall\frp\divides(p)}}
		\frac{\xi(\alpha)\xi_p^\e(\alpha)}{N(\fra^{-1}\alpha)^s}.
	\end{align*}
	Our assertion follows from the above calculations.
\end{proof}

%%%%%%%%%%%%%%%%%%%%%%%%%%%%%%%%%%%%%%%%%%%%%%%%%%
%
\subsection{Construction of the $p$-adic $L$-function}\label{subsection: padicHecke}
%
%%%%%%%%%%%%%%%%%%%%%%%%%%%%%%%%%%%%%%%%%%%%%%%%%%

In this subsection, we prove Theorem \ref{thm: padicL} and then construct the $p$-adic $L$-function 
associated to a Hecke character of the totally real field.
In what follows, we fix a set $\frC\subset\frI$ 
which represents the narrow ideal class group $\Cl^+_F(1)$. 
Recall that
by \eqref{eq: Cl(p^infty)}, we have a bijection 
$
    \coprod_{\fra\in\frC} \ol\Delta\backslash(\fra\otimes\bbZ_p)^\times
    \xrightarrow\cong\Cl^+_F(p^\infty).
$

\begin{proof}[Proof of Theorem \ref{thm: padicL}]
    Let $\chi\colon\Cl_F^+(\frg)\rightarrow\bbC^\times$ be a finite primitive Hecke character 
    of conductor $\frg$, where $\frg$ is an integral ideal of $\cO_F$
    which does not divide any power of $(p)$.
    In other words, $\frg$ is divisible by a prime ideal which is not a factor of $(p)$. 
    We will construct a measure $\mu_\chi$ on $\Cl^+_F(p^\infty)$ 
    which satisfies the interpolation property as in Theorem \ref{thm: padicL}. 
    
    Recall that, for any fractional ideal $\fra$ of $F$, 
    the Hecke character $\chi$ induces the function $\chi_\fra\colon\fra/\frg\fra\rightarrow\bbC$ 
    satisfying $\chi_\fra(\alpha)=\chi(\fra^{-1}\alpha)$ for any $\alpha\in\fra_+$.
    For $\fra\in\frC$, we let $\mu_{\chi_\fra}$ be the measure 
    on $\ol\Delta\setminus(\fra\otimes\bbZ_p)^\times$ given by
    \[
    	\mu_{\chi_\fra}\coloneqq\sum_{\xi\in\bbT_\prim^\fra[\frg]/\Delta}c_\chi(\xi)\mu_{\xiDelta},
    \]
    where $\bbT_\prim^\fra[\frg]$ denotes the set of primitive $\frg$-torsion point 
    in $\bbT^\fra[\frg]$.
    Note that since $\chi$ is primitive, we have 
    $\mu_{\chi_\fra}=\sum_{\xi\in\bbT^\fra[\frg]/\Delta}c_\chi(\xi)\mu_{\xiDelta}$
    by Proposition \ref{prop: Hecke primitive}.
    Then we let $\mu_\chi$ be the measure on $\Cl^+_F(p^\infty)$ induced by $\mu_{\chi_\fra}$
    through the bijection \eqref{eq: Cl(p^infty)}.
    In other words, we let 
    \[
    	\int_{\Cl^+_F(p^\infty)} f(\frx)d\mu_\chi(\frx)\coloneqq
    	\sum_{\fra\in\frC}\int_{\ol\Delta\backslash(\fra\otimes\bbZ_p)^\times}
    	f_\fra(x)d\mu_{\chi_\fra}(x)
    \]
    for any continuous function $f\colon\Cl^+_F(p^\infty)\rightarrow\bbC_p$,
    where we denote by $f_\fra$ the pull-back of $f$ to $\ol\Delta\backslash(\fra\otimes\bbZ_p)^\times$.
    %We denote by $N\colon\Cl_F^+(p^\infty)\rightarrow\bbZ_p^\times$ the norm map
    %given as the projective limit of maps $\Cl_F^+(p^n)\rightarrow\bbZ$ given by
    %mapping any fractional ideal $\fra\in\frI_p$ to $N(\fra)\in(\bbZ/p^n\bbZ)^\times$.

    Let $\chi_p$ be a finite character of $\Cl^+_F(p^\infty)$, 
    and suppose $\chi_p$ factors through $\Cl^+_F(p^n)$. 
	The finite Fourier inversion formula gives the equalites
	\begin{align*}
		\chi_\fra(x) &=\sum_{\xi\in\bbT_\prim^\fra[\frg]/\Delta}
		c_\chi(\xi)\xiDelta(x), &
		\chi_{p,\fra}(x)&=\sum_{\xi_p\in\bbT^\fra[p^n]/\Delta}c_{\chi_p}(\xi_p)\xipDelta(x)
	\end{align*}
	as functions on $\ol\Delta\backslash(\fra\otimes\bbZ_p)$ for any $\fra\in\frC$.
	Then for any $k\in\bbN$, we have
	\begin{align*}
		\int_{\Cl^+_F(p^\infty)} \chi_p(\frx)&\wh{N}(\frx)^k d\mu_\chi(\frx) 
		=\sum_{\fra\in\frC}
		\int_{\ol\Delta\backslash(\fra\otimes\bbZ_p)^\times}
	    \chi_{p,\fra}(x)\wh{N}_\fra(x)^k d\mu_{\chi_\fra}(x)\\
	    &=\sum_{\fra\in\frC}
	    \sum_{\substack{\xi\in\bbT^\fra_\prim[\frg]/\Delta\\\xi_p\in\bbT^\fra[p^n]/\Delta}}
	    c_\chi(\xi)c_{\chi_p}(\xi_p)
	    \int_{\ol\Delta\backslash(\fra\otimes\bbZ_p)^\times}
		\xipDelta(x)\wh{N}_\fra(x)^k d\mu_{\xiDelta}(x)\\
		&=\sum_{\fra\in\frC}
		\sum_{\substack{\xi\in\bbT^\fra_\prim[\frg]/\Delta\\\xi_p\in\bbT^\fra[p^n]/\Delta}}
		c_\chi(\xi)c_{\chi_p}(\xi_p)
	    \cL^\p(\xiDelta\,\xipDelta,-k)\\
	    &=\Biggl(\prod_{\frp\divides(p)}\bigl(1-\chi\!\chi_p(\frp)N\frp^{-k}\bigr)\Biggr)
	    L(\chi\!\chi_p,-k).
	\end{align*}
	The last equality follows from the fact that we have
	\begin{align*}
	\sum_{\fra\in\frC}
	\sum_{\substack{\xi\in\bbT^\fra_\prim[\frg]/\Delta\\\xi_p\in\bbT^\fra[p^n]/\Delta}}
	& c_\chi(\xi)c_{\chi_p}(\xi_p)
	\cL^\p(\xiDelta\,\xipDelta,s)\\
	&=
	\sum_{\fra\in\frC}\sum_{\substack{\xi\in\bbT^\fra_\prim[\frg]/\Delta\\\xi_p\in\bbT^\fra[p^n]/\Delta}}
	\sum_{\substack{\alpha\in\Delta\setminus\fra_+\\\alpha\in(\fra\otimes\bbZ_p)^\times}}
	c_\chi(\xi)c_{\chi_p}(\xi_p)\xiDelta(\alpha)\xipDelta(\alpha)N(\fra^{-1}\alpha)^{-s}
	\\
	&= 
	\sum_{\fra\in\frC}\sum_{\substack{\alpha\in\Delta\setminus\fra_+\\\alpha\in(\fra\otimes\bbZ_p)^\times}}
	\chi(\fra^{-1}\alpha)\chi_p(\fra^{-1}\alpha)N(\fra^{-1}\alpha)^{-s}\\
	&= 
	\sum_{\substack{\fra\subset\cO_F\\(\fra,(p))=1}}
	\chi(\fra)\chi_p(\fra)N\fra^{-s}
	=
	\Biggl( \prod_{\frp\divides(p)}\bigl(1-\chi\!\chi_p(\frp)N\frp^{-s}\bigr)\Biggr)L(\chi\!\chi_p,s)
	\end{align*}
	noting that $\alpha\in(\fra\otimes\bbZ_p)^\times$ implies that the ideal $\fra^{-1}\alpha$
	is prime to $\frp$ for any prime ideal of $F$ such that $\frp\divides(p)$.
\end{proof}

We define the $p$-adic $L$-function as follows.
Note that we have $\bbZ_p^\times=\mu\times(1+\bsp\bbZ_p)$, where $\mu$ is
the maximal torsion subgroup of $\bbZ_p^\times$ and $\bsp=p$ if $p\neq 2$ and $\bsp=p^2$ if $p=2$.
For $x\in\bbZ_p^\times$, we denote the first and second factors by
$\omega_p(x)\in\mu$ and $\langle x\rangle\in 1+\bsp\bbZ_p$.
We denote again by $\omega_p$ the character 
$\omega_p\circ\wh{N}\colon \Cl^+_F(p^\infty)\rightarrow\mu$. 

\begin{definition}\label{def: padicL}
	Let $\chi\colon\Cl^+_F(\frg)\rightarrow\bbC^\times$ be a finite primitive 
	Hecke character of conductor $\frg$, where $\frg$ does not divide any power of $(p)$.
	We define the $p$-adic $L$-function $L_p(\chi,s)$ by
	\[
		L_p(\chi,s)\coloneqq\int_{\Cl_F^+(p^\infty)}
		\omega_p(\frx)^{-1}\langle\wh{N}(\frx)\rangle^{-s} d\mu_\chi(\frx)
	\]
	for any $s\in\bbC_p$.
\end{definition}

As a corollary of Theorem \ref{thm: padicL}, the $p$-adic $L$-function satisfies the following.

\begin{corollary}
	For any integer $k\in \bbN$, we have
	\[
		L_p(\chi,-k) = 
		\Biggl( \prod_{\frp\divides(p)}
		\bigl(1-\chi\omega_p^{-k-1}(\frp)N\frp^{k}\bigr)\Biggr)L(\chi\omega_p^{-k-1},-k).
	\]
\end{corollary}

%%%%%%%%%%%%%%%%%%%%%%%%%%%%%%%%%%%%%%%%%%%%%%%%%%
%
%
%
\section{Proof of the Main Theorem}\label{section: main theorem}
%
%
%
%%%%%%%%%%%%%%%%%%%%%%%%%%%%%%%%%%%%%%%%%%%%%%%%%%

In this section, we will prove our main theorem.  More precisely, we will calculate the values of the $p$-adic polylogarithm
at torsion points of $\wh U_K$ and relate it to the special values of $p$-adic Hecke $L$-functions.

%%%%%%%%%%%%%%%%%%%%%%%%%%%%%%%%%%%%%%%%%%%%%%%%%%
%
\subsection{$p$-adic Polylogarithms and Specializations at Torsion Points}
%
%%%%%%%%%%%%%%%%%%%%%%%%%%%%%%%%%%%%%%%%%%%%%%%%%%

In this subsection, we will define the $p$-adic polylogarithm 
$\Li^\p_k(t)\in H^{g-1}\bigl(\wh U_K/F^\times_+,\sO_{\wh\bbT_K}\bigr)$
for $k\in\bbZ$ and calculate its specializations at torsion points.

For any $k\in\bbZ$, let $k^I\coloneqq(k,\ldots,k)\in\bbZ^I$.  Then we have the following.

\begin{lemma}\label{lem: twist}
	For any $k\in\bbZ$ and $\fra\in\frI$, consider the multiplication by $N\fra^k$ 
	on $\sO_{\wh{\bbT}^\fra_K}$, regarded as a homomorphism 
	$\sO_{\wh{\bbT}^\fra_K}(k^I)\to\sO_{\wh{\bbT}^\fra_K}$ 
	of sheaves on $\wh{\bbT}^\fra_K$. 
    Then the collection for all $\fra\in\frI$ gives an isomorphism 
	\begin{equation}\label{eq: isom45}
		\sO_{\wh\bbT_K}(k^I)\isomto\sO_{\wh\bbT_K}
	\end{equation}
	of $F^\times_+$-equivariant sheaves on $\wh\bbT_K=\amalg_{\fra\in\frI}\wh\bbT_K^\fra$.
\end{lemma}

\begin{proof}
	For any $\fra\in\frI$ and $x\in F^\times_+$, we have a commutative diagram
	\[
		\xymatrix{
		\bra{x}^*\sO_{\wh\bbT_K^{\fra}}(k^I)	 \ar[rr]^{N(\fra)^{k}}_\cong\ar[d]_{N(x)^{-k}}^\cong	
			&  &\bra{x}^*\sO_{\wh\bbT_K^{\fra}}\ar[d]_{\id}^\cong\\
			\sO_{\wh\bbT_K^{x\fra}}(k^I)	 \ar[rr]^{N(x\fra)^{k}}_\cong	& &\sO_{\wh\bbT_K^{x\fra}},
		}
	\]
	of sheaves on $\wh\bbT_K^{x\fra}$, which gives our assertion.
\end{proof}

%Using the isomorphism of Lemma \ref{lem: twist}, we may define 
%the $p$-adic polylogarithm $\Li^\p_{k}(t)$ in $H^{g-1}(\wh U_K/F^\times_+,\sO_{\wh\bbT_K})$
%for any integer $k\in\bbZ$ as follows.

\begin{definition}\label{def: k}
	For any $k\in\bbZ$, we define the $p$-adic polylogarithm $\Li^\p_{k}(t)$ to be the
	class in $H^{g-1}\bigl(\wh U_K/F^\times_+,\sO_{\wh\bbT_K}\bigr)$
	corresponding to the class $\Li^\p_{k^I}(t)$ through the isomorphism
	\[
		H^{g-1}\bigl(\wh U_K/F^\times_+,\sO_{\wh\bbT_K}\bigl(k^I\bigr)\bigr)
		\cong H^{g-1}\bigl(\wh U_K/F^\times_+,\sO_{\wh\bbT_K}\bigr)
	\]
	induced from the isomorphism \eqref{eq: isom45} of Lemma \ref{lem: twist}.	
	By construction, $\Li^\p_{k}(t)$ is given by the cocycle
	\begin{equation*}%\label{eq: modification}
	\Bigl(\sgn(\bsalpha)N\fra^k\Li^\p_{k^I,\sigma_\bsalpha}(t)\Bigr)
	\in
	\Biggl(\prod_{\fra\in\frI}\prod^\alt_{\bsalpha\in\sA_\fra^g}
	\Gamma\bigl(\wh U^\fra_{\bsalpha K},\sO_{\wh\bbT^\fra_K}\bigr)\Biggr)^{F^\times_+}.
	\end{equation*}
\end{definition}

Let $\xi\in\bbT(\ol\bbQ)$ be a torsion point lying in $\wh U_K$, and regard it as the rigid analytic space $\Sp K(\xi)$ with trivial action of $\Delta_\xi$.
Furthermore, for any cone $\sigma=\sigma_\bsalpha$ of dimension $g$ 
with $\bsalpha\in\sA_\fra^g$, we let $\wh U_{\sigma K}^\fra\coloneqq \wh U^\fra_{\bsalpha K}$.
We fix a Shintani decomposition $\Phi_{\xi}$ as in Lemma \ref{lem: SD} so that 
$\xi$ lies in $\wh U^\fra_{\sigma K}$ for any $\sigma\in\Phi_{\xi}$.
For any $\alpha\in\sA_\fra$, we let $V_\alpha=\wh U^\fra_\alpha\cap\xi$.
Then $\frV_\fra\coloneqq\{V_\alpha\}_{\alpha\in\sA_\fra}$ gives an affinoid covering of $\xi$.
We denote by $\sA_\xi$ the subset of $\sA_\fra$ consisting of $\alpha$ 
such that $\xi$ lies in $\wh U^\fra_{\alpha K}$.
Then the following lemma is proved by the same argument as in \cite{BHY19}*{Proposition 5.5}.

\begin{lemma}\label{lem: isomorphism}
	Let $\eta\in H^{g-1}(\xi/\Delta_\xi,\sO_\xi)$ be represented by a cocycle 
	\[
		 (\eta_\bsalpha)\in C^{g-1}(\frV_\fra/\Delta_\xi,\sO_\xi)
		=\biggl(\prod_{\bsalpha\in \sA_\xi^{g}}^\alt
		K(\xi)\biggr)^{\Delta_\xi}.
	\]	
	For any cone $\sigma\in\Phi_{\xi}$, let 
	$\eta_\sigma\coloneqq \sgn(\bsalpha)\eta_{\bsalpha}$ for a generator 
	$\bsalpha\in A^g_\xi$ of $\sigma$. 
	Then the homomorphism mapping the cocycle $(\eta_\bsalpha)$ to
	$
		\sum_{\sigma\in\Delta_\xi\backslash\Phi_{\xi}} \eta_{\sigma}
	$
	induces a canonical isomorphism
	\begin{equation}\label{eq: isomorphism}
	H^{g-1}(\xi/\Delta_\xi, \sO_\xi)\cong K(\xi).
	\end{equation}
\end{lemma}

Now we may define the value of the $p$-adic polylogarithm $\Li^\p_k(t)$ at $\xi$ as follows.

\begin{definition}
	Let $\xi$ be a torsion point lying in $\wh U_K$.
	For any $k\in\bbZ$,
	we denote by $\Li^\p_k(\xi)$ the image of $\Li^\p_k(t)$ with respect
	to the specialization map 
	\[
		H^{g-1}\bigl(\wh U_K/F_+^\times,\sO_{\wh\bbT_K}\bigr)
		\rightarrow H^{g-1}(\xi/\Delta_\xi,\sO_\xi)
	\]
	induced from the equivariant morphism $\xi\rightarrow \wh U_K$.
	We view $\Li^\p_k(\xi)$ as an element in $K(\xi)$ through 
	the isomorphism \eqref{eq: isomorphism},
	and call it the \textit{value} of the polylogarithm at the point $\xi$.
\end{definition}

The following result is crucial in the proof of our main theorem.

\begin{theorem}\label{thm: crucial}
	Let $\xi$ be a torsion point lying in $\wh U^\fra_K$ for $\fra\in\frI$.
	Then we have
	\[
		\Li^\p_{k}(\xi)=\int_{\ol\Delta\backslash(\fra\otimes\bbZ_p)^\times}
		\wh{N}_\fra(x)^{-k}d\mu_{\xiDelta}(x)
	\]
	for any integer $k\in\bbZ$.
\end{theorem}

To prove Theorem \ref{thm: crucial}, we need a formula for the value of the $p$-adic 
polylogarithm function $\Li^\p_{k^I,\sigma}(t)$ at $\xi$ for each cone $\sigma\in\Phi_\xi$, 
which is shown by using the technique of \cite{CdS88}*{\textbf{5.6.} Lemma}.
In what follows, we note that 
the multiplication by a torsion point $\xi\in\bbT^\fra(\ol\bbQ)$ induces a morphism 
$\varrho_\xi\colon\wh\bbT^\fra_{K(\xi)}\rightarrow\wh\bbT^\fra_{K(\xi)}$ 
given by $\varrho_\xi^*(t^\alpha)=\xi(\alpha)t^\alpha$.

\begin{proposition}\label{prop: crucial}
	Let $\fra\in\frI$, $\bsalpha\in\sA_\fra^g$, and $\sigma=\sigma_\bsalpha$.
	For any torsion point $\xi$ lying in $\wh U^\fra_{\sigma K}$ and integer $k\in\bbZ$,
	we have
	\[
		\Li^\p_{k^I,\sigma}(\xi)
		=\int_{(\fra\otimes\bbZ_p)^\times}
		N(x)^{-k}d\mu_{\sigma,\xi}(x),
	\]
	where $N(x)$ for $x\in(\fra\otimes\bbZ_p)^\times$ is defined as in \eqref{eq: N(x)}.
\end{proposition}

\begin{proof}
	Let $P\coloneqq\{\frp\in\Spec\cO_F\mid\frp\divides(p)\}$ be the
	set of prime ideals of $\cO_F$ dividing $(p)$ and let
	$\frp_J\coloneqq\prod_{\frp\in J}\frp$ for any $J\subset P$.
	We first prove that as functions in $\Gamma(\wh U^\fra_{\bsalpha K(\xi)}, \sO_{\wh\bbT^\fra_{K(\xi)}})$, we have
	\begin{equation}\label{eq: equality 44}
		\Li^\p_{0^I,\sigma}(t)
		=\cG^{\fra,\p}_\sigma(t)
		\coloneqq \sum_{J\subset P}\frac{(-1)^{|J|}}{N\frp_J}
		\sum_{\xi_{\frp_J}\in\bbT^\fra[\frp_J]}
		\varrho_{\xi_{\frp_J}}^*\cG^\fra_\sigma(t)
	\end{equation}
	Note that both sides of \eqref{eq: equality 44} lie in 
	$\wh B_\bsalpha\otimes K(\xi)\subset\Gamma(\wh U^\fra_{\bsalpha K(\xi)}, \sO_{\wh\bbT^\fra_{K(\xi)}})$, 
	where $B_\bsalpha$ is as in \eqref{eq: B alpha}.
	If we view the functions through the natural injection
	$
		\wh B_\bsalpha\otimes K(\xi)
		\hookrightarrow K(\xi)\llbracket t^\alpha\mid\alpha\in\fra_+\rrbracket,
	$
	then we have
	\begin{align*}
		\Li^\p_{0^I,\sigma}(t)
		&=\sum_{\substack{\alpha\in\breve\sigma\cap\fra\\\alpha\in(\fra\otimes\bbZ_p)^\times}}t^\alpha,&
		\varrho_{\xi_\frp}^*\cG^\fra_{\sigma}(t)&
		=\sum_{\alpha\in\breve\sigma\cap\fra}\xi_\frp(\alpha)t^\alpha
	\end{align*}	
	for any $\xi_\frp\in\bbT^\fra[\frp]$.
	Then \eqref{eq: equality 44} follows from the fact that 
	$\sum_{\xi_\frp\in\bbT^\fra[\frp]}\xi_\frp(\alpha)=N\frp$ if 
	$\alpha\in\frp$ and $\sum_{\xi_\frp\in\bbT^\fra[\frp]}\xi_\frp(\alpha)=0$ otherwise.
	For any integer $n\geq0$, if we apply the differential operator 
	$\partial\coloneqq\prod_{\tau\in I}\partial_\tau$ $n$ times to \eqref{eq: equality 44}, 
	and evaluate both sides at $\xi$, then we have
	\[
		\Li^\p_{-n^I,\sigma}(\xi)=\partial^n\cG^{\fra,\p}_\sigma(\xi) 
		= \partial^n\wt\cG^{\fra,\p}_{\sigma,\xi}(T)\big|_{T=0}
		=\int_{(\fra\otimes\bbZ_p)^\times}N(x)^n d\mu_{\sigma,\xi}(x),
	\]
	where the last equality follows from the construction of the measure $\mu_{\sigma,\xi}$.
	This is our assertion for integers $k\leq 0$. 
	
	Next let $k$ be a positive integer. 
	Recall the factorization $N(x)=N\fra\cdot\widehat{N}_\fra(x)$ (see \eqref{eq: N_a}). 
	Note that $\widehat{N}_\fra(x)\in\bbZ_p^\times$ for $x\in(\fra\otimes\bbZ_p)^\times$. 
	Hence we have 
	\[
		\widehat{N}_\fra(x)^{-k} = \lim_{r\rightarrow\infty}\widehat{N}_\fra(x)^{-k+(p-1)p^r}
	\]
	uniformly as functions on $(\fra\otimes\bbZ_p)^\times$.  
	This shows that 
	\[
		\Li^\p_{k^I,\sigma}(t) = \lim_{r\rightarrow\infty} N\fra^{-(p-1)p^r}\Li^\p_{(k-(p-1)p^r)^I,\sigma}(t) 
	\]
	and 
	\[
		\int_{(\fra\otimes\bbZ_p)^\times} N(x)^{-k} d\mu_{\sigma,\xi}(x)
		=\lim_{r\rightarrow\infty} N\fra^{-(p-1)p^r}\int_{(\fra\otimes\bbZ_p)^\times}
		N(x)^{-k+(p-1)p^r} d\mu_{\sigma,\xi}(x). 
	\]
	Hence our assertion for the case of $k>0$ follows from the result for $k\leq 0$.
\end{proof}

We now prove Theorem \ref{thm: crucial}.

\begin{proof}[Proof of Theorem \ref{thm: crucial}]
	By construction of the polylogarithm class
	\[
		\Li^\p_k(t)\in H^{g-1}(\wh U^\fra_K/\Delta,\sO_{\wh\bbT^\fra_K})
	\]
	given in Definition \ref{def: k}, 
	we see that the class $\Li^\p_k(\xi)$ in $H^{g-1}(\xi/\Delta_\xi,\sO_\xi)$
	is represented by the cocycle 
	\[
		\bigl(\sgn(\bsalpha)N\fra^k \Li^\p_{k^I,\sigma_\bsalpha}(\xi)\bigr)
		 \in C^{g-1}(\frV_\fra/\Delta_\xi,\sO_\xi).
	\]
	Then by Lemma \ref{lem: isomorphism}, this class maps to
	\[
		N\fra^k\sum_{\sigma\in\Delta_\xi\backslash\Phi_{\xi}} \Li^\p_{k^I,\sigma}(\xi)  \in K(\xi)
	\]
	through the isomorphism \eqref{eq: isomorphism}.
    By Proposition \ref{prop: crucial} and the definition of $\mu_{\xi\Delta}$, we have 
	\begin{align*}
	    \sum_{\sigma\in\Delta_\xi\backslash\Phi_{\xi}} \Li^\p_{k^I,\sigma}(\xi)
	    &=\sum_{\sigma\in\Delta_\xi\backslash\Phi_{\xi}}
	    \int_{(\fra\otimes\bbZ_p)^\times}N(x)^{-k}d\mu_{\sigma,\xi}(x)\\
	    &=\int_{\ol\Delta\backslash(\fra\otimes\bbZ_p)^\times}N(x)^{-k}d\mu_{\xiDelta}(x).
	\end{align*}
	This proves our assertion, since $\wh{N}_\fra(x)=N\fra^{-1}N(x)$ for 
	any $x\in\fra\otimes\bbZ_p$. 
\end{proof}

%%%%%%%%%%%%%%%%%%%%%%%%%%%%%%%%%%%%%%%%%%%%%%%%%%
%
\subsection{Proof of the Main Theorem}
%
%%%%%%%%%%%%%%%%%%%%%%%%%%%%%%%%%%%%%%%%%%%%%%%%%%

In this subsection, we will prove our main result, Theorem \ref{thm: main}.
We will then prove Corollary \ref{cor: main}, which coincides with Theorem \ref{thm: 2} of \S1.
The following result will be used to relate the $p$-adic $L$-functions to the $p$-adic polylogarithms.

\begin{lemma}\label{lem: new}
    Let $\frg$ be an integral ideal of $\cO_F$ which  does not divide any power of $(p)$, 
	$\chi\colon\Cl^+_F(\frg)\rightarrow\bbC^\times$ a primitive 
	Hecke character of conductor $\frg$, and $\chi_p$ a finite character of $\Cl^+_F(p^\infty)$. 
	Then we have $\chi_p\cdot\mu_\chi=\mu_{\chi\!\chi_p}$, 
    as an equality of measures on $\Cl^+_F(p^\infty)$ 
    (recall that $\chi\!\chi_p$ denotes the primitive Hecke character 
    induced by the product of $\chi$ and $\chi_p$; see Remark \ref{rem: chi chi_p}). 
\end{lemma}

\begin{proof}
    By Theorem \ref{thm: padicL}, we have 
    \[
        \int_{\Cl^+_F(p^\infty)}\chi_p(\frx)\chi'_p(\frx)\wh{N}(\frx)^k d\mu_\chi(\frx)
        =\Biggl(\prod_{\frp\divides(p)}\bigl(1-\chi\!\chi_p\chi'_p(\frp)N\frp^{-k}\bigr)\Biggr)
        L(\chi\!\chi_p\chi'_p,-k)
    \]
    for any finite character $\chi'_p$ of $\Cl^+_F(p^\infty)$ and $k\in\bbN$. 
    This shows that $\chi_p\cdot\mu_\chi$ satisfies the same interpolation property 
    as $\mu_{\chi\!\chi_p}$, hence our assertion follows. 
\end{proof}

We may now prove our main theorem.

\begin{theorem}\label{thm: main}
	Let $\frg$ be an integral ideal of $\cO_F$ which does not divide any power of $(p)$,and let 
	$\chi\colon\Cl^+_F(\frg)\rightarrow\bbC^\times$ be a primitive 
	Hecke character of conductor $\frg$.
	If we denote by $L_p(\chi \omega_p^{1-k},s)$ the $p$-adic $L$-function 
	associated to the Hecke character
	$\chi\omega_p^{1-k}$, then we have
	\[
		L_p(\chi\omega_p^{1-k},k)=
		\sum_{\xi\in\sT_\prim[\frg]} c_{\chi}(\xi)\Li_k^\p(\xi)
	\]
	for any integer $k\in\bbZ$, where $\sT_\prim[\frg]$ is the set of 
	primitive elements of $\sT[\frg]$ given before Proposition \ref{prop: Hecke primitive}.
\end{theorem}

\begin{proof}	
	By definition of the $p$-adic $L$-function given in Definition \ref{def: padicL} and 
	Lemma \ref{lem: new}, we have
	\begin{align*}
		L_p(\chi\omega_p^{1-k},k)&=
		\int_{\Cl^+_F(p^\infty)}\omega^{-1}_p(\frx) \langle \wh{N}(\frx)\rangle^{-k}
		d\mu_{\chi\omega_p^{1-k}}(\frx)
		=
		\int_{\Cl^+_F(p^\infty)}\omega^{-k}_p(\frx)\langle\wh{N}(\frx)\rangle^{-k}
		d\mu_{\chi}(\frx)\\
		&=
		\int_{\Cl^+_F(p^\infty)}\wh{N}(\frx)^{-k}
		d\mu_{\chi}(\frx)
		=
		\sum_{\fra\in\frC}
		\int_{\ol\Delta\backslash(\fra\otimes\bbZ_p)^\times}\wh{N}_\fra(x)^{-k}d\mu_{\chi}(x)
	\end{align*}
	for any integer $k\in\bbZ$, where $\frC$ is a set of representatives of $\Cl^+_F(1)$. 
	Moreover, since all $\xi\in\bbT^\fra_\prim[\frg]$ lie in $\wh U^\fra_K$, we have
	\begin{align*}
		\int_{\ol\Delta\backslash(\fra\otimes\bbZ_p)^\times}\wh{N}_\fra(x)^{-k}d\mu_{\chi}(x)
		&=\sum_{\xi\in\bbT^\fra_\prim[\frg]/\Delta_{\xi}}c_{\chi}(\xi) 	\int_{\ol\Delta\backslash(\fra\otimes\bbZ_p)^\times}\wh{N}_\fra(x)^{-k}d\mu_{\xiDelta}(x)\\
		&=\sum_{\xi\in\bbT^\fra_\prim[\frg]/\Delta}c_\chi(\xi)\Li^\p_k(\xi),
	\end{align*}
	for each $\fra\in\frC$, where the last equality follows from Theorem \ref{thm: crucial}. 
\end{proof}

We may reinterpret our main result in terms of the Gauss sum as follows.

\begin{definition}
    Let $\fra$ be a fractional ideal.
    For a Hecke character $\chi\colon\Cl^+_F(\frg)\to\bbC^\times$ 
    and a torsion point $\xi\in\bbT^\fra[\frg]$, 
    we define the Gauss sum $g(\chi,\xi)$ by
    \begin{equation*}%\label{eq: Gauss sum}
	g(\chi,\xi)\coloneqq
	N\frg\cdot c_{\chi}(\xi)=\sum_{\beta\in\fra/\frg\fra}\chi_\fra(\beta)\xi(-\beta).
    \end{equation*}
    By Lemma \ref{lem: two}, $g(\chi,\xi)$ depends only on the class of $\xi$ in $\sT[\frg]$.
\end{definition}

Before proceeding to our result, we first define a certain action of $\Cl^+_F(\frg)$ 
on the set $\sT[\frg]$. 
Let $\frI_\integ$ be the subset of $\frI$ consisting of integral ideals of $F$,
viewed as a submonoid of $\frI$ with respect to the multiplication,
and let $\frb\in\frI_\integ$.
Then for any $\fra\in\frI$, the natural inclusion $\fra\frb\subset\fra$
induces a morphism of algebraic tori
$
	\varphi(\frb)\colon\bbT^{\fra} \rightarrow \bbT^{\fra\frb}.
$
If $\xi\colon\fra\rightarrow R^\times$ is a point in $\bbT^\fra$, then 
$\varphi(\frb)(\xi)\colon\fra\frb\rightarrow R^\times$ 
is given by $\varphi(\frb)(\xi)(\alpha)=\xi(\alpha)$ for any $\alpha\in\fra\frb\subset\fra$.
Then the morphisms $\varphi(\frb)$ for all $\fra\in\frI$ defines a morphism
$\varphi(\frb)\colon\bbT\rightarrow\bbT$, 
and we obtain a homomorphism $\varphi\colon\frI_\integ\rightarrow\End(\bbT)$ of monoids, 
where $\End(\bbT)$ denotes the monoid of endomorphisms of $\bbT$ as a groupoid.

\begin{lemma}\label{lem: well-defined}
	For any $\frb\in\frI_\integ$, the morphism $\varphi(\frb)$ 
	commutes with the action of $F^\times_+$, hence induces a map
	\begin{equation}\label{eq: T}
		\varphi(\frb)\colon\sT\rightarrow\sT.
	\end{equation}
\end{lemma}

\begin{proof}
	let $x\in F_+^\times$ and $\fra\in\frI$.
	We have a commutative diagram 
	\[
		\xymatrix{
			\fra \frb \ar@{^{(}->}[r]\ar[d]_\cong^x& \fra \ar[d]_\cong^x\\
			x\fra\frb\ar@{^{(}->}[r] &  x\fra		
		}	
	\]
	where the horizontal arrows are the inclusions and the vertical arrows 
	are multiplication by $x$.  
	This shows that the induced maps on the tori gives a commutative diagram
	\[
		\xymatrix{
			\bbT^{\fra\frb} & \bbT^{\fra} \ar[l]_{\varphi(\frb)}\\
			\bbT^{x\fra\frb}\ar[u]^{\bra{x}}_\cong & \bbT^{x\fra}\ar[u]^{\bra{x}}_\cong\ar[l]_{\varphi(\frb)},
		}	
	\]
	hence $\varphi(\frb)$ induces the map \eqref{eq: T} on $\sT$.
\end{proof}

Lemma \ref{lem: well-defined} gives a morphism of monoids 
$\varphi\colon\frI_\integ\rightarrow\End(\sT)$.
This induces, for each integral ideal $\frg$, an action of $\Cl^+_F(\frg)$ on $\sT[\frg]$ 
as follows. 

\begin{lemma}\label{lem: WD1}
	Let $\beta\in \cO_{F_+}$ such that $\beta\equiv1\pmod\frg$, and put $\frb=(\beta)$.
	Then for any $\xi\in\sT[\frg]$, we have $\varphi(\frb)(\xi)=\xi$. 
	In other words, $P_+(\frg)\cap\frI_\integ$ acts trivially on $\sT[\frg]$, hence 
	$\Cl^+_F(\frg)\cong(\frI_\frg\cap \frI_\integ)/(P_+(\frg) \cap \frI_\integ)$ 
	acts via $\varphi$ on the set $\sT[\frg]$. 
\end{lemma}

\begin{proof}
	Consider $\xi\in\bbT^\fra[\frg]$ for some fractional ideal $\fra\in\frI$.
	By definition, $\varphi(\frb)(\xi)$ is the homomorphism
	$\varphi(\frb)(\xi)\colon \fra\frb/\frg\fra\frb \rightarrow\ol\bbQ^\times$
	given by $\varphi(\frb)(\xi)(\alpha\beta)=\xi(\alpha\beta)$ for any $\alpha\in\fra$.
	On the other hand, $\beta^{-1}\in F_+^\times$ defines a map 
	$\bra{\beta^{-1}}\colon\bbT^\fra\rightarrow\bbT^{\beta\fra}$.
	Then $\bra{\beta^{-1}}(\xi)$ is an element in $\bbT^{\beta\fra}=\bbT^{\fra\frb}$
	such that $\bra{\beta^{-1}}(\xi)(\alpha\beta)=\xi(\beta^{-1}\alpha\beta)
	 =\xi(\alpha)$ for any $\alpha\in\fra$.
	Since $\beta\equiv1\pmod\frg$, we have $\xi(\alpha)=\xi(\alpha\beta)$ 
	for any $\alpha\in\fra$, hence $\varphi(\frb)(\xi)=\bra{\beta^{-1}}(\xi)$ 
	as functions on $\fra\frb$.
	Thus $\xi$ and $\varphi(\frb)(\xi)$ coincides as elements in $\sT[\frg]$, as desired. 
\end{proof}

In the following, we put $\xi^\frb\coloneqq \varphi(\frb)(\xi)$ 
for $\xi\in\bbT^\fra_\prim[\frg]$ and $\frb\in\Cl^+_F(\frg)$, 
where $\varphi$ is the action given in Lemma \ref{lem: WD1}. 
We have the following.

\begin{lemma}\label{lem: ST}
    The action of $\Cl^+_F(\frg)$ 
    given by Lemma \ref{lem: WD1}
    is simply transitive on $\sT_\prim[\frg]$.
\end{lemma}

\begin{proof}
    Take any two elements of $\sT_\prim[\frg]$ 
    represented by $\xi\in\bbT^\fra_\prim[\frg]$ and $\xi'\in\bbT^{\fra'}_\prim[\frg]$ 
    respectively. Then, for an integral ideal $\frb$ prime to $\frg$, 
    $\xi^\frb$ and $\xi'$ represent the same element in $\sT_\prim[\frg]$ 
    if and only if there exists $x\in F^\times_+$ such that $\bra{x}(\xi^\frb)=\xi'$, 
    i.e., $\fra\frb=x\fra'$ and $\xi'(\alpha)=\xi(x\alpha)$ holds for all $\alpha\in\fra'$. 
    Thus we have to show that, for arbitrarily given $\xi$ and $\xi'$, 
    there exists a pair of $\frb$ and $x$ as above, and $\frb$ is unique up to $P_+(\frg)$. 
    
    We construct such $\frb$ and $x$. First, take an integral ideal $\frb_0$ prime to $\frg$ 
    from the narrow ideal class of $\fra^{-1}\fra'$, so that there exists $x_0\in F^\times_+$ 
    satisfying $\fra\frb_0=x_0\fra'$. Then define a character 
    $\xi'_0\in\bbT^{\fra'}_\prim[\frg]$ by $\xi'_0(\alpha)=\xi(x_0\alpha)$. 
    Since $\fra'/\frg\fra'$ is a free $\cO_F/\frg$-module of rank one, 
    so is its character group $\bbT^{\fra'}[\frg]$, 
    and $\bbT^{\fra'}_\prim[\frg]$ is the set of generators of the latter. 
    Hence the two elements $\xi',\xi'_0\in\bbT^{\fra'}_\prim[\frg]$ coincide up to 
    multiplication by $(\cO_F/\frg)^\times$, that is, 
    there exists $\alpha_0\in(\cO_F/\frg)^\times$ such that $\xi_0(\alpha)=\xi'_0(\alpha_0\alpha)$. 
    Then the pair $(\frb,x)\coloneqq (\alpha_0\frb_0,\alpha_0 x)$ has the desired property. 
    
    To show the uniqueness, we may assume that $\xi=\xi'$. Then we have $\frb=(x)$ with $x\in\cO_{F+}$. 
    Moreover, by the above interpretation of $\bbT^\fra_\prim[\frg]$, 
    the identity $\xi'(\alpha)=\xi(x\alpha)$ implies that $x\equiv 1\bmod{\frg}$. 
    Hence $\frb$ is trivial in $\Cl^+_F(\frg)$. This completes the proof
    of our assertion. 
\end{proof}

The following lemma describes the behavior of the Gauss sums under this action. 

\begin{lemma}\label{lem: Gauss sums}
    Let $\xi\in\bbT^\fra_\prim[\frg]$ and $\frb\in\Cl^+_F(\frg)$. 
    Then for any Hecke character $\chi$ of conductor $\frg$, we have 
    \[
        g(\chi,\xi^\frb)=g(\chi,\xi)\chi(\frb)^{-1}. 
    \]
\end{lemma}

\begin{proof}
    We may compute both sides by choosing a representative $\frb$ of the given ray class 
    to be integral. Then we have 
    \[
    	g(\chi,\xi^\frb)
    	=\sum_{\alpha\in\fra\frb/\frg\fra\frb}\chi_{\fra\frb}(\alpha)\xi^\frb(-\alpha)
    	=\chi(\frb)^{-1}\sum_{\alpha\in\fra/\frg\fra}\chi_{\fra}(\alpha)\xi(-\alpha)
    	=g(\chi,\xi)\chi(\frb)^{-1},
    \]
    where we have used the fact that since $\frb$ is prime to $\frg$,
    the natural inclusion $\fra\frb\subset\fra$ induces an isomorphism
    $\fra\frb/\frg\fra\frb\cong\fra/\frg\fra$. 
\end{proof}

Now Theorem \ref{thm: main} gives the following corollary, 
which was stated as Theorem \ref{thm: 2} in \S1.

\begin{corollary}\label{cor: main}
	Let the notation be as in Theorem \ref{thm: main}, and fix an arbitrary $\xi\in\sT_\prim[\frg]$. 
	Then we have 
	\[
		L_p(\chi\omega_p^{1-k},k)
		=\frac{g(\chi,\xi)}{N\frg}
		\sum_{\frb\in\Cl^+_F(\frg)}
		\chi(\frb)^{-1}\,\Li_k^\p(\xi^\frb)
	\]
	for any integer $k\in\bbZ$.
\end{corollary}

\begin{proof}
    The assertion follows from Theorem \ref{thm: main}, Lemma \ref{lem: ST}, and Lemma \ref{lem: Gauss sums}.
\end{proof}

\begin{remark}\label{rem: speculation}
	The action $\varphi$ of $\Cl^+_F(\frg)$ on $\sT[\frg]$ described in Lemma \ref{lem: WD1}
	gives a striking parallel to the case of elliptic curves with complex multiplication.
	Although we do not currently have any concrete ideas to attack the problem,
	if $\sT$ can be equipped with a certain $F$-structure so that the points of $\sT$
	have a natural action of $\Gal(\ol F/F)$, and if we could prove that this action
	is compatible with the action $\varphi$ on $\sT[\frg]$
	through the isomorphism $\Cl^+_F(\frg)\cong\Gal(F(\frg)/F)$ given by class field theory,
	then the stack $\sT=\bbT/F_+^\times$ 
	would provide a basis for considering Kronecker's Jugendtraum 
	in the case of totally real fields.
\end{remark}

\subsection*{Acknowledgement} 
The authors would like to thank the KiPAS program FY2014--2018 of the Faculty of Science and Technology 
at Keio University, especially Professors Yuji Nagasaka and Masato Kurihara, 
for providing an excellent environment making this research possible. 
The authors would also like to thank Hohto Bekki, Tatsuya Ohshita and Shinji Chikada for discussion.
The authors sincerely thank the editor and the referee for carefully reading the article and for comments which helped to improve the article.

\end{document}